\documentclass[a4paper,11pt]{article}
\usepackage{amsmath,amssymb,amsthm,amscd}
\newtheorem{thm}{Theorem}[section]
\newtheorem{cor}[thm]{Corollary}

\newtheorem{prop}[thm]{Proposition}
\newtheorem{defn}[thm]{Definition}
\newtheorem{rem}[thm]{Remark}

\numberwithin{equation}{section}
\newcommand{\nt}{\noindent}
\newcommand{\st}{\subset}
\newcommand{\R}{\mathbb{R}}

\newcommand{\N}{\mathbb{N}}
\newcommand{\K}{\mathbb{K}}

\newcommand{\C}{\mathbb{C}}

\newcommand{\Ccal}{\mathcal{C}}
\newcommand{\Kcal}{\mathcal{K}}
\newcommand{\Acal}{\mathcal{A}}
\newcommand{\Dcal}{\mathcal{D}}
\newcommand{\Scal}{\mathcal{S}}
\newcommand{\Ecal}{\mathcal{E}}
\newcommand{\Lcal}{\mathcal{L}}
\newcommand{\Hcal}{\mathcal{H}}

\newcommand{\Pcal}{\mathcal{P}}
\newcommand{\Wcal}{\mathcal{W}}

\newcommand{\Ucal}{\mathcal{U}}
\newcommand{\Rb}{\overline{\R}}
\newcommand{\dbar}{{d\hspace{-0,05cm}\bar{}\hspace{0,05cm}}}
\newcommand{\Op}{\textup{Op}}
\newcommand{\cl}{\textup{(cl)}}
\newcommand{\clw}{\textup{cl}}
\newcommand{\loc}{\textup{loc}}
\newcommand{\cone}{\textup{cone}}
\newcommand{\comp}{\textup{comp}}
\newcommand{\y}{\infty}
\newcommand{\la}{\langle}
\newcommand{\ra}{\rangle}
\newcommand{\imb}{\textup{Im}}
\newcommand{\reb}{\textup{Re}}
\newcommand{\ur}{\textup{r}}
\newcommand{\id}{\textup{id}}
\newcommand{\beq}{\begin{equation}}
\newcommand{\eeq}{\end{equation}}

\begin{document}

\title{The singular functions of branching edge asymptotics}
\author{B.-W. Schulze\footnote{B.-Wolfgang Schulze, Institute of Mathematics, University of Potsdam, 14469, Potsdam, Germany; e-mail: schulze@math.uni-potsdam.de}, L. Tepoyan\footnote{L. Tepoyan, Yerevan State University, Department of Mathematics and Mechanics, A. Manoogian str. 1,0025 Yerevan, Armenia; e-mail: tepoyan@yahoo.com}}

\pagestyle{headings}
\maketitle
\date{}




\centerline{\thanks{University of Potsdam,  supported by DFG GZ : SCHU 808/22-1}}
\begin{abstract}
We investigate the structure of branching asymptotics appearing in solutions to elliptic edge problems. The exponents in powers of the half-axis variable, logarithmic terms, and coefficients depend on the variables on the edge and may be branching.
\end{abstract}


\tableofcontents
\newpage
\section*{Introduction}
\addcontentsline{toc}{section}{Introduction}
      \markboth{INTRODUCTION}{INTRODUCTION}


The solutions to elliptic problems on a manifold with edge are expected to have asymptotics of the form
\beq\label{as}
u(r,x,y) \sim  \sum_{j=0}^J\sum_{k=0}^{m_j}
c_{jk}(x,y)r^{-p_j}\textup{log}^k r
\eeq
as $r\to 0$, with exponents $p_j\in\C$, and $m_j\in\N$ ($=\{0,1,2,\ldots\}$). Here $(r,x,y)$ are the variables in an open stretched wedge ${\R}_+\times X\times\Omega$ for a closed smooth manifold $X$ of dimension $n$ and an open set $\Omega\subseteq {\R}^q$. If the respective operator is a differential operator of the form
\beq\label{AO}
 A=r^{-\mu
}\sum_{j+|\alpha|\leq\mu} a_{j\alpha}(r,y)(-r\partial
_r)^j(rD_y)^\alpha
\eeq
for coefficients  $a_{j\alpha}(r,y)\in C^{\y} \big(\overline{\R}_+\times\Omega ,\textup{Diff}^{\,\mu-(j+|\alpha|)}(X)\big)$ (with  $\textup{Diff}^{\,\nu}(X)$ being the space of differential operators of order $\nu$ on $X$) then the asymptotic data
\beq\label{data}
\Pcal :=\{(p_j,m_j)\}_{0\le j\le J}\subset\C\times\N,
 \eeq
$J=J(\Pcal)\in\N\cup\{{\y}\}$ are known to be determined by the leading conormal symbol
\beq\label{cono}
\sigma_{\textup{c}}(A)(y,z) :=\sum_{j=0}^\mu
a_{j0}(0,y)z^j,
\eeq
regarded as a family of differential operators
\beq\label{cont}
\sigma_{\textup{c}}(A)(y,z):H^s(X)\to
H^{s-\mu}(X)
\eeq
 of order $\mu$, smooth in $y\in\Omega$ and holomorphic in $z$. In the elliptic case it is known that the operators \eqref{cont} are parameter-dependent elliptic of order $\mu$ where the parameter is $\imb\,z$ with $z$ varying on a so-called weight line
  \beq\label{Gb}
 {\Gamma}_{\beta}:=\{z\in \C : \reb \,z =\beta\}
 \eeq
for every real $\beta$.

It is well-known that for any fixed $y\in\Omega$ the operators \eqref{cont} are bijective for all $z$ off some discrete set $D(y)\subset\C$, where $D(y)\cap\{c<\reb\,z<c'\}$ is finite for every $c\le c'$, cf. Bleher \cite{Bleh1}. Those non-bijectivity points are just responsible for the exponents $-p_j$ in \eqref{as}. More precisely, ${\sigma}_{\textup{c}}^{-1}(A)(y,z)$ is an $L_{\clw}^{-\mu}(X)$-valued meromorphic function with poles at the points $p_j$ of (finite) multiplicities $m_j+1$ and finite rank Laurent coefficients in $L^{-\y}(X)$ at the powers $(z-p_j)^{-(k+1)}, 0\le k\le m_j$. Here $L_{\clw}^{\nu}(X), \nu\in\R$, means the space of all classical pseudo-differential operators on $X$ of order $\nu$, and $L^{-\y}(X) :=\bigcap_{\nu\in\R}L^{\nu}_{\clw}(X)$ is the space of smoothing operators.

If $\sigma_{\textup{c}}(A)$ is independent of $y$ we have constant discrete edge asymptotics of solutions, cf. the terminology below. Even in this case it is interesting to observe the nature of coefficients $c_{jk}$ in \eqref{as} depending on the considered Sobolev smoothness $s\in\R$ of the solutions. The Sobolev smoothness of the coefficients $c_{jk}$ in $y$ also depends on $\reb\,p_j$.
Clearly in general the leading conormal symbol $\sigma_{\textup{c}}(A)$ depends on $y$ and then also the set $D(y)$. In this case the inverse $\sigma_{\textup{c}}^{-1}(A)$ is a $y$-dependent family of meromorphic operator functions with poles $p_j(y)$ varying in the complex plane and possible branchings where the multiplicities $m_j(y)+1$ may have jumps, including the above-mentioned Laurent coefficients. These effects have been studied in a number of papers, cf. \cite{Schu34}, \cite{Schu36} and \cite{Schu67}. In particular, also the Sobolev smoothness in $y$ of the coefficients $c_{jk}(x,y)$ is branching. The program is going on, and in the present article we study some features of the functional analytic structure of singular functions in the variable branching case which are not yet analyzed so far.

The characterization of asymptotics of solutions to singular PDE-problems is a central issue of solvability theory of elliptic equations on a singular configuration. One of the classical papers in this connection is \cite{Kond1} of Kondratyev on boundary value problems on manifolds with conical singularities. Since then there appeared numerous investigations in this field, also on boundary problems for operators without the transmission property, or mixed and transmission problems, see, in particular, Eskin's book \cite{Eski2}. The present investigation is dominated by the pseudo-differential approach to generate asymptotics via parametrices and elliptic regularity, see, in particular, the monographs \cite{Schu2}, \cite{Egor1}, \cite{Haru13}, and the references there. Note that a similar philosophy applies also for corner singularities where asymptotics appear in iterated form, cf. \cite{Schu27}, or,  the recent investigations, \cite{Haba1}, \cite{Schu74}.

This paper is organized as follows.

First in Section 1 we outline some necessary tools on constant discrete edge asymptotics in the frame of weighted edge spaces and corresponding subspaces. We then pass to a more detailed investigation of the singular functions and show some essential simplification compared with other expositions, say, \cite{Egor1} or \cite{Schu20}, namely, that the cut-off functions may be chosen independently of the edge covariable $\eta$, modulo edge-flat remainders. We do that including the so-called continuous asymptotics.

In Section 2 we consider  variable branching edge asymptotics, formulated in terms of smooth functions with values in analytic functionals that are pointwise discrete and of finite order. Basics and tools can be found in \cite{Schu2}, \cite{Kapa10}; the notion itself has been first established in \cite{Schu34}, \cite{Schu36} and further studied in detail in \cite{Schu67}. Here we show a refinement of a result of \cite{Schu20} on the representation of singular functions with variable continuous asymptotics by analytic functionals without explicit dependence on the edge variable $y$. In particular, the preparations from Section 1 on  $\eta$-independent cut-off functions  allow us to find the claimed  new representation in a unique way. We finally apply this result to the case of variable branching asymptotics and obtain the surprising effect that the pointwise discrete behaviour in $y$ may be shifted into a new functional that gives rise to a  localization of Sobolev smoothness of ``coefficients of asymptotics'', both in variable branching as well as in continuous asymptotics.

\section{The constant discrete edge asymptotics} 

\subsection{Edge spaces and specific operator-valued symbols} 

Let us first recall what we understand by abstract edge spaces modelled on a space with group action.

First if this space is a Hilbert space $H$ such a group action is a family $\kappa=\{\kappa_{\lambda}\}_{\lambda\in{\R}_+}$ of isomorphisms $\kappa_{\lambda}:H\to H$ with $\kappa_{\lambda}\kappa_{{\lambda}'}=\kappa_{\lambda{\lambda}'}$ for all $\lambda, {\lambda}'\in\R_+$, and $\lambda\to \kappa_{\lambda}h$ represents a function in $C({\R}_+,H)$ for every $h\in H$. As is known we have an estimate
\beq\label{estk}
\|\kappa_{\lambda}\|_{\Lcal(H)}\le c\big(\max\{\lambda, {\lambda}^{-1}\}\big)^M
\eeq
for all $\lambda\in{\R}_+$, for some constants $c>0, M>0$, depending on $\kappa$ (a proof may be found in \cite{Hirs2}).
We also need the case of a Fr\'{e}chet space $E$ written as a projective limit
$\underset{\overleftarrow{j\in\N}}{\lim} E^j$  of Hilbert spaces, with continuous embeddings $E^j\hookrightarrow E^0$ for all $j$, where $E^0$ is endowed with a group action $\kappa$ and $\kappa|_{E^j}$ defines a group action in $E^j$ for every $j$. The constants $c$ and $M$ in \eqref{estk} then may depend on $j$.
Now $\Wcal^s({\R}^q,H), s\in\R$, is defined to be the completion of $\Scal({\R}^q,H)$ with respect to the norm
\beq\label{N}
\|u\|_{{\Wcal}^s({\R}^q,H)}=\left\{\int\la\eta\ra^{2s}\|{\kappa}^{-1}_{\la\eta\ra}\hat{u}(\eta)\|^2_H\,d\eta\right\}^{1/2}
\eeq
with $\hat{u}(\eta)=(F_{y\to\eta}u)(\eta)$ being the Fourier transform, ${\la\eta\ra}=(1+|\eta|^2)^{1/2}$. For a Fr\'{e}chet space
$E=\underset{\overleftarrow{j\in\N}}{\lim} E^j$
we have ${\Wcal}^s({\R}^q,E^j)$, $j\in\N$, and we set
$$
{\Wcal}^s({\R}^q,E)=\lim_{\overleftarrow{j\in\N}} {\Wcal}^s({\R}^q,E^j).
$$
Recall that we obtain an equivalent norm to \eqref{N} when we replace $\la\eta\ra$ by a function $\eta\to[\eta]$, strictly positive, smooth, with $[\eta]=|\eta|$ for $|\eta|>C$ for some $C>0$.

In the general discussion we often consider the Hilbert space case; the generalization to Fr\'{e}chet spaces will be obvious. Observe that ${\Wcal}^s({\R}^q,H)\subset{\Scal}'({\R}^q,H)$. For an open set $\Omega\subseteq{\R}^q$ by ${\Wcal}^s_{\loc}(\Omega,H)$ we denote the space of all $u\in{\Dcal}'(\Omega,H)$ such that $\varphi u\in{\Wcal}^s({\R}^q,H)$ for every $\varphi\in C_0^{\y}(\Omega)$. Moreover, ${\Wcal}^s_{\comp}(\Omega,H)$ denotes the subspace of all elements of ${\Wcal}^s({\R}^q,H)$ that have compact support in $\Omega$.
Clearly the spaces $\Wcal^s({\R}^q,H)$ depend on the choice of $\kappa$. If necessary we write $\Wcal^s({\R}^q,H)_{\kappa}$ in order to indicate the specific group action $\kappa$. The case $\kappa=\id$ for all $\lambda\in{\R}_+$ is always admitted. Then we have
$$
\Wcal^s({\R}^q,H)_{\id}=H^s({\R}^q,H)
$$
which is the standard Sobolev space of $H$-valued distributions.
Observe that
\beq\label{cap}
{\bigcap}_{s\in\R}{\Wcal}^s({\R}^q,H)_{\kappa} :={\Wcal}^{\y}({\R}^q,H)_{\kappa}={\Wcal}^{\y}({\R}^q,H)_{\id},
\eeq
i.e., the dependence on $\kappa$ disappears when $s=\y$. This is a consequence of \eqref{estk}.
From the definition we have an isomorphism
\beq\label{K}
\K=:F^{-1}\kappa_{[\eta]}F :\Wcal^s({\R}^q,H)_{\id}\to \Wcal^s({\R}^q,H)_{\kappa}
\eeq
for every $s\in\R$, in particular,
$$
\K: \Wcal^{\y}({\R}^q,H)_{\id}\to \Wcal^{\y}({\R}^q,H)_{\id}.
$$
We employ spaces $\Wcal^s({\R}^q,H)$ for certain Hilbert spaces $H$ based on the Mellin transform.

 The analysis on a singular manifold refers to a large extent to the
 Mellin transform
 $$
 Mu(z)=\int_0^{\y}r^{z-1}u(r)\,dr
 $$
 first for $u\in C_0^{\y}({\R}_+)$ and then extended to various
 distribution spaces, also vector-valued ones. For $u\in
 C_0^{\y}({\R}_+)$ we obtain an entire function in the complex
 $z$-plane. Function/distribution spaces on
 ${\Gamma}_{\beta}$ always refer to $\rho=\imb \,z$ for
 $z\in{\Gamma}_{\beta},$ e.g. the Schwartz space ${\mathcal
 S}({\Gamma}_{\beta})$ or $L^2({\Gamma}_{\beta})$ with respect to
 the Lebesgue measure on ${\R}_{\rho}$. Recall that the Mellin transform induces a
 continuous operator $M:C_0^{\y}({\R}_+)\to {\Acal}(\C)$ with ${\Acal}(\C)$ being the space of entire functions
 in $z$. In particular, for $u\in
 C_0^{\y}({\R}_+)$ we can form the weighted Mellin transform  $M_{\gamma}:C_0^{\y}({\R}_+)\to {\Scal
 }({\Gamma}_{1/2-\gamma})$ of
 weight $\gamma\in\R$,  defined as
 $M_{\gamma}u:=Mu|_{{\Gamma}_{1/2-\gamma}}$. As is well-known, $M_{\gamma}$
 extends to an isomorphism
 $M_{\gamma}:r^{\gamma}L^2({\R}_+)\to L^2({\Gamma}_{1/2-\gamma}),$ and then
 $$
(M_{\gamma}^{-1}g)(r)=\int_{{\Gamma}_{1/2-\gamma}}r^{-z}g(z)\,\dbar
z
 $$
 for $\dbar z=(2\pi i)^{-1}dz.$ Analogously as standard Sobolev
 spaces based on $L^2$-norms and the Fourier transform we can form
 weighted Mellin Sobolev spaces
 ${\Hcal}^{s,\gamma}({\R}_+\times{\R}^n)$ as the completion of
 $C_0^{\y}({\R}_+\times{\R}^n)$ with respect to the norm
 $$
 \|u\|_{{\Hcal}^{s,\gamma}({\R}_+\times{\R}^n)}=\left\{\int_{{\Gamma}_{(n+1)/2-\gamma}}\int_{{\R}^n}{\la
 z,\xi\ra}^{2s}\left|(M_{\gamma-n/2, r\to z}F_{x\to \xi}u)(z,\xi)\right|^2\,\dbar zd\xi \right\}^{1/2},
 $$
 with $F=F_{x\to \xi}$ being the Fourier transform in ${\R}^n$.
Moreover, if $X$ is a smooth closed manifold of dimension $n$ we have analogous
spaces
${\Hcal}^{s,\gamma}(X^{\wedge})$ for
$$
X^{\wedge}:={\R}_+\times X
$$
based on the local spaces ${\Hcal}^{s,\gamma}({\R}_+\times{\R}^n)$
and defined with the help of charts and a partition of unity on $X$. Note that (in our notation)  the meaning of $\gamma$ depends on the dimension $n$. In the case $s=\y$ we have a canonical identification
\beq\label{inf}
{\Hcal}^{\y,\gamma}(X^{\wedge})={\Hcal}^{\y,\gamma-n/2}({\R}_+){\hat{\otimes}}_{\pi}C^{\y}(X)\cong C^{\y}\big(X,{\Hcal}^{\y,\gamma-n/2}({\R}_+)\big);
\eeq
here ${\hat{\otimes}}_{\pi}$ means the projective tensor product between the respective spaces.

In this exposition a cut-off function $\omega$ on the half-axis is
any $\omega\in C_0^{\y}({\Rb}_{+})$ that is equal to $1$ close to
$0$. It will be essential also to employ the spaces
\beq\label{Ks}
{\Kcal}^{s,\gamma}(X^{\wedge}):=\{\omega u+(1-\omega)v:u\in
{\Hcal}^{s,\gamma}(X^{\wedge}), v\in H^s_{\cone}(X^{\wedge})\}.
\eeq
Here $H^s_{\cone}(X^{\wedge})$ is defined as follows. Choose
any diffeomorphism ${\chi}_1:U\to V$ from a coordinate
neighbourhood $U$ on $X$ to an open set $V\subset S^n$ (the unit
sphere in ${\R}^{n+1}_{\tilde{x}}$), and let $\chi:{\R}_{+}\times
U\to \Gamma:=\{\tilde{x}\in{\R}^{n+1}\setminus \{0\}:
{\tilde{x}}/|{\tilde{x}}|\in V\}$ be defined by
$\chi(r,x):=r{\chi}_1(x), r\in{\R}_{+}$. Then
$H^s_{\cone}(X^{\wedge})$ is the set of all $v\in
H^s_{\loc}(\R\times X)|_{{\R}_{+}\times X}$ such that for any
$\varphi\in C_0^{\y}(U)$ we have $((1-\omega)\varphi v)\circ
{\chi}^{-1}\in H^s({\R}^{n+1})$, for every coordinate
neighbourhood $U$ on $X$. Concerning more details on those spaces,
cf.~\cite{Schu20} or \cite{Schu27}. In particular,
${\Hcal}^{s,\gamma}(X^{\wedge})$ and
${\Kcal}^{s,\gamma}(X^{\wedge})$ are Hilbert spaces in suitable
scalar products, and we have
${\Hcal}^{0,0}(X^{\wedge})={\Kcal}^{0,0}(X^{\wedge})=r^{-n/2}L^2({\R}_+\times
X)$ with $L^2$ referring to $drdx$ and $dx$  associated
with a fixed Riemannian metric on $X$,  $n=\dim X$. Analogously as \eqref{inf} we also have
\beq\label{infn}
{\Kcal}^{\y,\gamma}(X^{\wedge})={\Kcal}^{\y,\gamma-n/2}({\R}_+)\hat{\otimes}_{\pi}C^{\y}(X)=
C^{\y}\big(X,{\Kcal}^{\y,\gamma-n/2}({\R}_+)\big).
\eeq
Here ${\Kcal}^{\y,\gamma-n/2}({\R}_+)$ is endowed with its natural Fr\'{e}chet topology.
In
order to formulate asymptotics of elements in
${\Kcal}^{s,\gamma}(X^{\wedge})$ we first fix so-called weight
data $(\gamma,\Theta)$ for $\gamma\in\R$ and
$\Theta=(\vartheta,0],-\y\le\vartheta<0.$ Define the Fr\'{e}chet
space
$$
{\Kcal}^{s,\gamma}_{\Theta}(X^{\wedge})=\lim_{\overleftarrow{k\in\N}}
{\Kcal}^{s,\gamma-\vartheta-(1+k)^{-1}}(X^{\wedge})
$$
of elements of flatness $\Theta$ relative to $\gamma$. For purposes below we also introduce the spaces ${\Kcal}^{s,\gamma;e}(X^{\wedge}):= \la r\ra^{-e}{\Kcal}^{s,\gamma}(X^{\wedge})$, ${\Kcal}^{s,\gamma;e}_{\Theta}(X^{\wedge}):= \la r\ra^{-e}{\Kcal}^{s,\gamma}_{\Theta}(X^{\wedge})$ for any $s, \gamma, e\in\R$. In order to define subspaces with asymptotics we consider a sequence
\beq\label{pair}
\mathcal P=\{(p_j,m_j)\}_{j=0,1,\ldots, J}\subset\C\times\N
\eeq
for a
$J=J(\Pcal)\in\N\cup\{\y\}$ such that
$(n+1)/2-\gamma+\vartheta<\reb\,p_j<(n+1)/2-\gamma$ for all $0\le
j\le J$, $J(\Pcal)<\y$ for $\vartheta>-\y$. In the case $\vartheta=-\y$ and $J=\y$ we assume $\reb\,p_j\to -\y$ as $j\to\y$. Such a $\Pcal$ will be called a discrete asymptotic type
associated with $(\gamma,\Theta)$. We set
${\pi}_{\C}\Pcal:=\{p_j\}_{j=0,1,\ldots, J}$. Observe that for any
$p\in\C, \reb\,p<(n+1)/2-\gamma$, and $c\in C^{\y}(X)$ we have
$\omega(r)c(x)r^{-p}{\textup{log}}^kr\in
{\Kcal}^{\y,\gamma}(X^{\wedge})$ for $k\in\N$ and any cut-off
function $\omega$. Given a discrete asymptotic type $\Pcal$ for finite $\Theta$ we form the
space
\beq\label{sing2}
{\Ecal}_{\Pcal}:=\{
\omega(r)\sum_{j=0}^J\sum_{k=0}^{m_j}c_{jk}r^{-p_{j}}{\textup{log}}^k r
:  c_{jk}\in C^{\y}(X) \},
\eeq
for some fixed cut-off function $\omega$. This space is Fr\'{e}chet
in a natural way (in fact, isomorphic to a corresponding direct sum
 of finitely many copies of $C^{\y}(X)$), and we have
 ${\Kcal}^{s,\gamma}_{\Theta}(X^{\wedge})\cap{\Ecal}_{\Pcal}=\{0\}.$
 Then the direct sum
 \beq\label{sub}
 {\Kcal}^{s,\gamma}_{\Pcal}(X^{\wedge}):={\Kcal}^{s,\gamma}_{\Theta}(X^{\wedge})+{\Ecal}_{\Pcal}
 \eeq
 is again a Fr\'{e}chet space. The spaces \eqref{sub} are examples of subspaces
 of ${\Kcal}^{s,\gamma}(X^{\wedge})$ with discrete
 asymptotics of type $\Pcal$. The definition can be easily extended to
 asymptotic types $\Pcal=\{(p_j,m_j)\}_{j=0,1,\ldots, J}$
 associated with $(\gamma, (-\y, 0])$ and $J\in\N\cup\{\y\}$.  In this case we form ${\Pcal}_k:=\{(p,m)\in\Pcal: \reb \,p>(n+1)/2-\gamma-(k+1)\},
 k\in\N$; then ${\Pcal}_k$ is finite and associated with $\big(\gamma,
 (-(k+1),0]\big)$. Thus we have the spaces
 ${\Kcal}^{s,\gamma}_{{\Pcal}_k}(X^{\wedge})$ and we set
 $$
{\Kcal}^{s,\gamma}_{\Pcal}(X^{\wedge}):=\lim_{\overleftarrow{k\in\N}}{\Kcal}^{s,\gamma}_{{\Pcal}_k}(X^{\wedge}).
$$
Another technical tool that we employ later on are operator-valued symbols based on twisted symbolic estimates. Let $H$ and $\tilde{H}$ be Hilbert spaces with group actions $\kappa$ and $\tilde{\kappa}$, respectively.

By $S^{\mu}(\Omega\times{\R}^q;H,\tilde{H})$  for an open set
$\Omega\subseteq {\R}^p$ we denote the set of all $a(y,\eta)\in
C^{\y}\big(\Omega\times{\R}^q, {\Lcal (H, \tilde{H})}\big)$ such that
\beq\label{ests}
\|\tilde{\kappa}_{[\eta]}^{-1}\{D_y^{\alpha}D_{\eta}^{\beta}a(y,
\eta)\}{\kappa}_{[\eta]}\|_{\Lcal (H, \tilde{H})}\le
c[\eta]^{\mu-|\beta|}
\eeq
for all $(y, \eta)\in K\times{\R}^q, K\Subset \Omega,$ and
$\alpha\in {\N}^p, \beta\in {\N}^q,$ for constants $c=c(\alpha,
\beta, K)>0.$ Such $a$ are called (operator-valued) symbols of order
$\mu$. For instance, if $a(y, \eta)$ is homogeneous of order $\mu$
for large $|\eta|$ then it is such a symbol. By
$S^{\mu}_{\clw}(\Omega\times{\R}^q; H, \tilde{H})$ we denote the
subspace of classical symbols, i.e., the set of those $a(y, \eta)\in
S^{\mu}(\Omega\times{\R}^q; H, \tilde{H})$ with an asymptotic
expansion into symbols that are homogeneous of order $\mu-j,
j\in\N$, for large $|\eta|$. Let $S^{(\mu)}\big(\Omega\times({\R}^q\setminus\{0\});H,\tilde{H}\big)$ be the space of those $a_{(\mu)}(y,\eta)\in C^{\y}\big(\Omega\times({\R}^q\setminus\{0\}),\Lcal(H,\tilde{H})\big)$ such that $a_{(\mu)}(y,\lambda\eta)={\lambda}^{\mu}{\tilde{\kappa}}_{\lambda}a(y,\eta){\kappa}_{\lambda}^{-1}$ for all $\lambda\in{\R}_+$.  Every $a(y, \eta)\in
S^{\mu}(\Omega\times{\R}^q; H, \tilde{H})$ has a principal symbol of
order $\mu$, i.e., the unique $a_{(\mu)}(y, \eta)\in
S^{(\mu)}\big(\Omega\times({\R}^q\setminus\{0\}); H, \tilde{H}\big)$ such that
$$
a(y, \eta)-\chi(\eta)a_{(\mu)}(y, \eta)\in
S^{\mu-1}_{\clw}(\Omega\times{\R}^q; H, \tilde{H})
$$
for any fixed excision function $\chi$.

If a consideration is valid in the classical as well as the general case we write as subscript $(\clw)$. If necessary we also write $S^{\mu}_{\cl}(\Omega\times{\R}^q; H,
\tilde{H})_{\kappa, \tilde{\kappa}}$ for the respective spaces of
symbols. The spaces of symbols with constant coefficients will be denoted by
$S^{\mu}_{\cl}({\R}^q; H, \tilde{H})$. The spaces
$S^{\mu}_{\cl}(\Omega\times{\R}^q; H, \tilde{H})$ are Fr\'{e}chet in
a natural way, $S^{\mu}_{\cl}({\R}^q; H, \tilde{H})$ are closed
subspaces, and we have
$$
S^{\mu}_{\cl}(\Omega\times{\R}^q; H,\tilde{H})=C^{\y}(\Omega,
S^{\mu}_{\cl}\big({\R}^q; H, \tilde{H})\big).
$$
In the case $p=2q$ and $\Omega\times\Omega$ for
$\Omega\subseteq {\R}^q$ instead of $\Omega\subseteq{\R}^p$ we also
write $(y,y')$ rather than $y$.

For every $a(y,y',\eta)\in S^{\mu}(\Omega\times\Omega\times{\R}^q;
H, \tilde{H})$ the operator $\textup{Op}(a):C_0^{\y}(\Omega, H)\to
C^{\y}(\Omega, \tilde{H})$, defined by
\beq\label{Osc}
{\Op}_y(a)u(y):=\int\!\!\!\int e^{i(y-y')\eta}a(y,y',\eta)u(y')\,dy'\dbar\eta,
\eeq
 extends to a continuous map
\beq\label{co}
\textup{Op}(a):{\Wcal}^s_{\comp}(\Omega, H)\to
{\Wcal}^{s-\mu}_{\textup{loc}}(\Omega, \tilde{H})
\eeq
for any $s\in\R$. The continuity \eqref{co} has been established in \cite[page 283]{Schu2} for all spaces $H,\tilde{H}$ that are of interest here. The case of general $H,\tilde{H}$ with group action was given in \cite{Seil1}. In the special case of $a(\eta)\in S^{\mu}({\R}^q; H, \tilde{H})$ the operator $\Op(a)$ induces a continuous operator
\beq\label{cont1}
\Op(a):{\Wcal}^s({\R}^q,H)\to {\Wcal}^{s-\mu}({\R}^q,\tilde{H})
\eeq
for any $s\in\R$. Here
\beq\label{NO}
\|\Op(a)\|_{{\Lcal}({\Wcal}^s({\R}^q,H),{\Wcal}^{s-\mu}({\R}^q,\tilde{H}))}\le \underset{\eta\in{\R}^q}{\sup} [\eta]^{-\mu}\|{\tilde{\kappa}}^{-1}_{[\eta]}a(\eta){\kappa}_{[\eta]}\|_{{\Lcal}(H,\tilde{H})}.
\eeq
\begin{rem}\label{S}
Observe that \eqref{cont1} already holds for $a(\eta)\in C^{\y}\big({\R}^q,{\Lcal}(H,\tilde{H})\big)$ when the $0$-th symbolic
estimate \eqref{ests} holds, namely,
$$
\|{\tilde{\kappa}}^{-1}_{[\eta]}a(\eta){\kappa}_{[\eta]}\|_{{\Lcal}(H, \tilde{H})}\le c[\eta]^{\mu}
$$
for all $\eta\in{{\R}^q}$, for some $c>0$.
\end{rem}
We will employ below a slight modification of such a construction. Let us start, in particular, with the case $H=\C$ with the trivial group action. Symbols in $S^{\mu}_{\cl}(\Omega\times{\R}^q; \C, \tilde{H})$ are also referred to as potential symbols. Consider, for instance, the case of symbols $a(\eta)$ with constant coefficiens, i.e., without $y$-dependence. Such symbols are realized as multiplications of $c\in\C$ by an element $f(\eta)\in\tilde{H}$. The symbolic estimates have the form
\beq\label{est1}
\|{\kappa}^{-1}_{[\eta]}D_{\eta}^{\beta}f(\eta)\|_{{\Lcal}(\C,\tilde{H})}=\|{\kappa}^{-1}_{[\eta]}D_{\eta}^{\beta}f(\eta)\|_{\tilde{H}}\le
 C[\eta]^{\mu-|\beta|}.
\eeq
In our applications we have the situation that for a Fr\'{e}chet space $E=\underset{\overleftarrow{j\in\N}}{\lim} E^j$ for Hilbert spaces $E^j$ and the trivial group action $\id$ on all $E^j$ we encounter $E$ and tensor products  $\tilde{H}\hat{\otimes}_{\pi}E$ rather than $\C$ and  $\tilde{H}$. In our case $E$ will be nuclear, and then we have
$$
\tilde{H}\hat{\otimes}_{\pi}E=\lim_{\overleftarrow{k\in\N}}\tilde{H}{\otimes}_{H}E^j
$$
with ${\otimes}_{H}$ being the Hilbert tensor product. Such things are well-known, but details may be found, e.g., in
\cite[page 38]{Remp3}.
From $f(\eta)\in S^{\mu}_{\cl}({\R}^q; \C, \tilde{H})$ we pass to the operator function $f(\eta)\otimes{\id}_E$. This can be interpreted as a symbol
$$
f\otimes{\id}_E\in S^{\mu}_{\cl}({\R}^q; E, \tilde{H}\hat{\otimes}_{\pi}E)=\lim_{\overleftarrow{j\in\N}}S^{\mu}_{\cl}({\R}^q; E_j, \tilde{H}{\otimes}_{H}E^j).
$$
In fact, instead of \eqref{est1}  we have the symbolic estimates
\begin{equation*}
\begin{split}
\|({\kappa}^{-1}_{[\eta]}\otimes{\id}_{E^j})D_{\eta}^{\beta}\big(f(\eta)\otimes{\id}_{E^j} \big)&\|_{{\Lcal}(E^j,\tilde{H}{\otimes}_H E^j)}=\|{\kappa}^{-1}_{[\eta]}D_{\eta}^{\beta}f(\eta)\|_{\tilde{H}}\|{\id}_{E^j}\|_{{\Lcal}(E^j,E^j)}\\
=&\|{\kappa}^{-1}_{[\eta]}D_{\eta}^{\beta}f(\eta)\|_{\tilde{H}}\le C[\eta]^{\mu-|\beta|}
\end{split}
\end{equation*}
for every $j$. Similarly as \eqref{cont1} we obtain continuous operators
\beq\label{cont3}
{\Op}_y(f\otimes {\id}_{E^j}):H^s({\R}^q,E^j)\to{\Wcal}^{s-\mu}({\R}^q,\tilde{H}{\otimes}_{H}E^j).
\eeq
The space on the right refers to the group action $\kappa_{\lambda}\otimes{\id}_E$, such that
$$
\|u\|_{{\Wcal}^t({\R}^q,\tilde{H}{\otimes}_HE^j)}=
\left\{\int[\eta]^{2t}\left\|({\kappa}^{-1}_{[\eta]}\otimes{\id}_{E^j})\hat{u}(\eta)\right\|_{\tilde{H}{\otimes}_HE^j}^2\,d\eta\right\}^{1/2}
$$
for every $j$. We have
$$
{\Wcal}^t({\R}^q,\tilde{H}\hat{{\otimes}}_{\pi}E)=\lim_{\overleftarrow{j\in\N}}{\Wcal}^t({\R}^q,\tilde{H}{\otimes}_{H}E^j),
$$
$t\in\R$, and it follows altogether
\beq\label{cont2}
{\Op}_y(f\otimes {\id}_E):H^s({\R}^q,E)\to{\Wcal}^{s-\mu}({\R}^q,\tilde{H}{\hat{\otimes}}_{\pi}E).
\eeq

\subsection{Characterization of singular functions} 

In order to formulate the singular functions of discrete edge asymptotics we endow the Fr\'{e}chet   spaces ${\Kcal}^{s, \gamma}_{\Pcal}(X^{\wedge})$ with the group action
\beq\label{group}
({\kappa}_{\lambda}u)(r,x) :={\lambda}^{(n+1)/2}u(\lambda r,x),
\eeq
$\lambda\in{\R}_+$. The larger spaces ${\Kcal}^{s, \gamma}(X^{\wedge})$ are endowed with this group action as well, and we may consider $\kappa_{\lambda}$ also over the spaces  ${\Kcal}^{s, \gamma}_{\Theta}(X^{\wedge})$ of functions of flatness $\Theta$ relative to $\gamma$. This allows us to define the spaces
$$
{\Wcal}^s\big({\R}^q,{\Kcal}^{s, \gamma}(X^{\wedge})\big)\supset{\Wcal}^s\big({\R}^q,{\Kcal}^{s, \gamma}_{\Pcal}(X^{\wedge})\big)\supset{\Wcal}^s\big({\R}^q,{\Kcal}^{s, \gamma}_{\Theta}(X^{\wedge})\big).
$$
The space ${\Ecal}_{\Pcal}$ of singular functions of cone asymptotics, defined for any fixed cut-off function $\omega$, is not invariant under $\kappa$. Nevertheless, according to \eqref{sub} it is desirable also to decompose ${\Wcal}^s\big({\R}^q,{\Kcal}^{s, \gamma}_{\Pcal}(X^{\wedge})\big)$ into a flat part, namely, ${\Wcal}^s\big({\R}^q,{\Kcal}^{s, \gamma}_{\Theta}(X^{\wedge})\big)$ and a subspace generated by the singular functions. Here we proceed as follows. We first look at the case $\kappa=\id$ and observe that from \eqref{sub} we have a direct sum
$$
{\Wcal}^s\big({\R}^q,{\Kcal}^{s, \gamma}_{\Pcal}(X^{\wedge})\big)_{\id}={\Wcal}^s\big({\R}^q,{\Kcal}^{s, \gamma}_{\Theta}(X^{\wedge})\big)_{\id}+{\Wcal}^s({\R}^q,{\Ecal}_{\Pcal})_{\id}.
$$
Clearly ${\Wcal}^s({\R}^q,{\Ecal}_{\Pcal})_{\id}$ is a subspace of ${\Wcal}^s\big({\R}^q,{\Kcal}^{s, \gamma}_{\Pcal}(X^{\wedge})\big)_{\id}$. According to \eqref{K} we have an isomorphism
$$
\K:{\Wcal}^s\big({\R}^q,{\Kcal}^{s, \gamma}_{\Pcal}(X^{\wedge})\big)_{\id}\to{\Wcal}^s\big({\R}^q,{\Kcal}^{s, \gamma}_{\Pcal}(X^{\wedge})\big)_{\kappa}.
$$
Thus, applying \eqref{K} to the subspace
${\Wcal}^s({\R}^q,{\Ecal}_{\Pcal})_{\id}$ we obtain a subspace of \newline ${\Wcal}^s\big({\R}^q,{\Kcal}^{s, \gamma}_{\Pcal}(X^{\wedge})\big)_{\kappa}$ and a direct decomposition
\beq\label{dec}
{\Wcal}^s\big({\R}^q,{\Kcal}^{s, \gamma}_{\Pcal}(X^{\wedge})\big)_{\kappa}={\Wcal}^s\big({\R}^q,{\Kcal}^{s, \gamma}_{\Theta}(X^{\wedge})\big)_{\kappa}+\K{\Wcal}^s({\R}^q,{\Ecal}_{\Pcal})_{\id}.
\eeq
By virtue of the definition of the operator $\K$ we have
\beq\label{sing}
\begin{split}
\K{\Wcal}^s({\R}^q&,{\Ecal}_{\Pcal})_{\id}=
\textup{span}\{F^{-1}_{y\to\eta}[\eta]^{(n+1)/2}\omega(r[\eta]){\hat{c}}_{jk}(x,\eta)(r[\eta])^{-p_j}{\log}^k(r[\eta])\\
&: 0\le k\le m_j,
j=0,1,\ldots,J, {\hat{c}}_{jk}(x,\eta)\in\hat{H}^s\big({\R}^q_{\eta}, C^{\y}(X)\big)\}
\end{split}
\eeq
for $\hat{H}^s\big({\R}^q_{\eta}, C^{\y}(X)\big):=F_{y\to\eta}H^s\big({\R}^q_y, C^{\y}(X)\big)$.
This follows from the fact that
\beq\label{si}
\begin{split}
{\Wcal}^s({\R}^q&,{\Ecal}_{\Pcal})_{\id}=H^s({\R}^q,{\Ecal}_{\Pcal})=\textup{span}\{\omega(r)c_{jk}(x,y)r^{-p_j}{\log}^kr\\
&: 0\le k\le m_j,j=0,1,\ldots,J, c_{jk}(x,y)\in{H}^s\big({\R}^q_y, C^{\y}(X)\big)\}.
\end{split}
\eeq
The explicit form \eqref{sing} gives us a first impression on the nature of singular terms of the edge asymptotics for a constant (in $y$) asymptotic type $\Pcal$ and finite $\Theta$.

Let us briefly comment the case $s=\y$ where the ${\Wcal}^s$-spaces do not depend on $\kappa$, cf. the relation \eqref{cap}. In that case we may choose the singular functions in the form \eqref{si} for $s=\y$, i.e., the $r$-powers, logarithmic terms and the cut-off function $\omega$ do not contain $\eta$. In other words we have the direct decomposition
\beq\label{A}
\begin{split}
{\Wcal}^{\y}\big({\R}^q, {\Kcal}_{\Pcal}^{\y,\gamma}(X^{\wedge})\big)_{\id}&=
{\Wcal}^{\y}\big({\R}^q, {\Kcal}_{\Theta}^{\y,\gamma}(X^{\wedge})\big)_{\id} +
{\Wcal}^{\y}({\R}^q, {\Ecal}_{\Pcal})_{\id}\\
&=H^{\y}\big({\R}^q, {\Kcal}_{\Theta}^{\y,\gamma}(X^{\wedge})\big)+H^{\y}({\R}^q,{\Ecal}_{\Pcal}).
\end{split}
\eeq
On the other hand we have
\beq\label{B}
{\Wcal}^{\y}\big({\R}^q, {\Kcal}_{\Pcal}^{\y,\gamma}(X^{\wedge})\big)_{\kappa}=
{\Wcal}^{\y}\big({\R}^q, {\Kcal}_{\Theta}^{\y,\gamma}(X^{\wedge})\big)_{\kappa}+
\K{\Wcal}^{\y}({\R}^q, {\Ecal}_{\Pcal})_{\id},
\eeq
cf. the relation \eqref{dec} for $s=\y$. By virtue of \eqref{cap} the only formal difference between \eqref{A} and \eqref{B} for $s=\y$ lies in the difference between $H^{\y}({\R}^q,{\Ecal}_{\Pcal})$
and $\K{\Wcal}^{\y}({\R}^q, {\Ecal}_{\Pcal})_{\id}$.

\begin{prop}\label{equiv}
Let $\Pcal=\{(p_j,m_j)\}_{j=0,1,\ldots,J}$   be a discrete asymptotic type associated with the weight data $(\gamma,\Theta)$ for finite $\Theta=(\vartheta,0]$. Then there is a direct decomposition
$$
{\Wcal}^{\y}\big({\R}^q, {\Kcal}_{\Pcal}^{\y,\gamma}(X^{\wedge})\big)_{\kappa}=
{\Wcal}^{\y}\big({\R}^q, {\Kcal}_{\Theta}^{\y,\gamma}(X^{\wedge})\big)_{\kappa} +
H^{\y}({\R}^q, {\Ecal}_{\Pcal})
$$
where
\beq\label{sin}
\begin{split}
H^{\y}({\R}^q,{\Ecal}_{\Pcal})&=\textup{span}\big\{\omega(r)c_{jk}(x,y)r^{-p_j}{\log}^k(r):
 0\le k\le m_j,\\ j&=0,1,\ldots, J, c_{jk}\in H^{\y}\big({\R}^q_y, C^{\y}(X)\big)\big\}.
\end{split}
\eeq
\end{prop}
\begin{proof}

 We write down once again
\eqref{sing} for $s=\y$, namely,
\beq\label{ex}
\begin{split}
\K{\Wcal}^{\y}({\R}^q,\,& {\Ecal}_{\Pcal})_{\id}=\textup{span}\big\{F^{-1}_{y\to\eta}[\eta]^{(n+1)/2}\omega(r[\eta]){\hat{c}}_{jk}(x,\eta)(r[\eta])^{-p_j}\\
{\log}^k(r[\eta]),\,& 0\le k\le m_j, j=0,1,\ldots, J, {\hat{c}}_{jk}(x,\eta)\in{\hat{H}}^{\y}\big({\R}^q_{\eta}, C^{\y}(X)\big)\big\}.
\end{split}
\eeq
First it is clear that $[\eta]^{-p_j}$ gives rise to a modification of the coefficients in ${\hat{H}}^{\y}\big({\R}^q_{\eta}, C^{\y}(X)\big)$, since $[\eta]^M{\hat{H}}^{\y}\big({\R}^q_{\eta}, C^{\y}(X)\big)={\hat{H}}^{\y}\big({\R}^q_{\eta}, C^{\y}(X)\big)$ for any real $M$. Moreover, writing $\log\,(r[\eta])=\log r +\log\, [\eta]$ we can dissolve ${\log}^k(r[\eta])$ as a sum of products between powers of $\log r$ and $\log [\eta]$. Also  the $\log\, [\eta]$-terms are absorbed by ${\hat{H}}^{\y}({\R}^q_{\eta}, C^{\y}(X)$, and hence we get rid of $[\eta]$ in \eqref{ex}, except for the cut-off function $\omega(r[\eta])$. In order to remove $[\eta]$ from the cut-off function we apply Taylor's-formula. Choose another cut-off function $\tilde{\omega}\succ\omega$ where $\tilde{\varphi}\succ\varphi$ or $\varphi\prec\tilde{\varphi}$ means that $\tilde{\varphi}$ is equal to $1$ on $\textup{supp}\,\varphi$ such that $\tilde{\omega}(r)\big(\omega(r[\eta])-\omega(r)\big)=\omega(r[\eta])-\omega(r)$ for all $r$ and $\eta$. Then
\beq\label{Mu}
\begin{split}
\omega(r[\eta])-\omega(r)&=\tilde{\omega}(r)\left\{ \frac{(r[\eta])^{N+1}}{N!}\int_0^1(1-t)^N\omega^{(N+1)}(r[\eta]t)\,dt\right.\\
&\left.-\frac{r^{N+1}}{N!}\int_0^1(1-t)^N\omega^{(N+1)}(rt)\,dt\right\}.
\end{split}
\eeq
If we verify that this function belongs to
${\Wcal}^{\y}\big({\R}^q, {\Kcal}_{\Theta}^{\y,\gamma}(X^{\wedge})\big)$ for sufficiently large $N$ we may replace in the formula \eqref{ex} $\omega(r[\eta])$ by $\omega(r)$, i.e., after the comments before on how to remove $[\eta]$ from $(r[\eta])^{-p_j}$ or ${\log}^k(r[\eta])$ we see altogether, that the singular functions of edge asymptotics for $s=\y$ are of the form \eqref{sin}. The fact that a function $\psi\in C_0^{\y}({\Rb}_+)$ of sufficiently high flatness at $r=0$, i.e., $r^{-N}\psi(r)\in C_0^{\y}({\Rb}_+)$ for large and fixed $N$, belongs to ${\Wcal}^{\y}\big({\R}^q,{\Kcal}^{\y,\gamma}_{\Theta}(X^{\wedge})\big)$, follows from the fact that $\psi(r)$ may be regarded as an operator-valued symbol
$$
\psi\in S^{\mu}({\R}^q;\C,\tilde{H}^j)
$$
for $\tilde{H}^j:={\Kcal}^{s,\gamma-n/2-\vartheta-(1+j)^{-1}}({\R}_+)$ and some $\mu=\mu(s)\in\R$, for all $j\in\N$.

 In fact, it is clear that $\psi\in\tilde{H}^j$ for a fixed sufficiently large $N\in\N$. Moreover, $\{\kappa_{\lambda}\}_{\lambda\in{\R}_+}$ defined by \eqref{group} acts on $\tilde{H}^j$ for every $j$. Thus, by virtue of \eqref{estk} we have
 $\|\kappa_{\lambda}\|_{\Lcal({\tilde{H}}^j)}\le c\,{\max\{\lambda, {\lambda}^{-1}\}}^M$ for constants $c,M>0$ depending on the space ${\tilde{H}}^j$, in fact, on $s$. The symbolic estimates \eqref{est1} for $\psi$ rather than $f(\eta)$, here independent of $\eta$, reduce to the estimate for $\beta=0$, and we have
 $$
 \|{\kappa}_{[\eta]}^{-1}\psi\|_{\Lcal(\C,{\tilde{H}}^j)}\le \|{\kappa}_{[\eta]}^{-1}\|_{\Lcal({\tilde{H}}^j)}\|\psi\|_{\Lcal(\C,{\tilde{H}}^j)}\le c[\eta]^{\mu}\|\psi\|_{{\tilde{H}}^j}
 $$
  for some $\mu$ and  constants $c=c(j)>0$. Then, writing $E:=C^{\y}(X)=\underset{\overleftarrow{j\in\N}}{\lim} E^j$, where we may take $E^j:= H^j(X)$, we obtain   $\psi\otimes{\id}_{E^j}\in S^{\mu}({\R}^q;E^j,{\tilde{H}}^j{\otimes}_{H}E^j)$. This gives us the continuity
 $$
 {\Op}_y(\psi\otimes {\id}_{E^j}):H^{\tilde{s}}({\R}^q,E^j)\to{\Wcal}^{{\tilde{s}}-\mu}({\R}^q,{\tilde{H}}^j{\otimes}_{H}E^j)
 $$
for every $\tilde{s}\in\R$, cf. \eqref{cont3} which entails
$$
 {\Op}_y(\psi\otimes {\id}_{E^j}):H^{\y}({\R}^q,E^j)\to{\Wcal}^{\y}({\R}^q,{\tilde{H}}^j{\otimes}_{H}E^j)
 $$
and
\begin{equation*}
\begin{split}
 {\Op}_y\big(\psi\otimes {\id}_{C^{\y}(X)}\big):H^{\y}\big({\R}^q,C^{\y}(X)\big)\to {\Wcal}^{\y}&\big({\R}^q,{\Kcal}_{\Theta}^{\y,\gamma-n/2}({\R}_+)\big){\hat{\otimes}}_{\pi}C^{\y}(X)\\
 = {\Wcal}^{\y}&\big({\R}^q,{\Kcal}_{\Theta}^{\y,\gamma}(X^{\wedge})\big).
 \end{split}
 \end{equation*}
Here we employed the relation ${\Kcal}_{\Theta}^{\y,\gamma}(X^{\wedge})={\Kcal}_{\Theta}^{\y,\gamma-n/2}({\R}_+){\hat{\otimes}}_{\pi}C^{\y}(X)$. For the second summand in \eqref{Mu} we argue as follows. The function  $g(r):=\int_0^1(1-t)^N{\omega}^{(N+1)}(rt)\,dt$ on ${\Rb}_+$ belongs to $C^{\y}({\Rb}_+)$ and is bounded on ${\Rb}_+$ including all its $r$-derivatives. The same is true of $f(\eta)=g(r[\eta])$ as a function in $r\in{\Rb}_+$. The notation $f(\eta)$ indicates that $f$ is regarded as an operator-valued symbol. The operator of multiplication by $g(r)$ induces continuous operators
$g: {\tilde{H}}^j\to {\tilde{H}}^j$ for all $j$. Thus
${\Op}_y(f):{\Wcal}^s({\R}^q,{\tilde{H}}^j)\to {\Wcal}^s({\R}^q,{\tilde{H}}^j)$
is continuous for every $s\in\R$, cf. Remark \ref{S}. Setting $h(\eta)=f(\eta){\tilde{\omega}}(r)(r[\eta])^{N+1}/{N!}$ and $\psi(r)=\tilde{\omega}(r){r^{N+1}}/{N!}$ we have
$$
{\Op}_y(h\otimes{\id}_{E^j})=[\eta]^{N+1}{\Op}_y(f\otimes{\id}_{E^j}){\Op}_y(\psi\otimes{\id}_{E^j}).
$$
From the first step of the proof we know that
$$
\psi\otimes{\id}_{E^j}\in S^{\mu+N+1}({\R}^q;E^j,{\tilde{H}}^j{\otimes}_H{\tilde{H}}^j).
$$
It follows altogether
$$
{\Op}_y(h)\otimes{\id}_{E^j}:H^{\tilde{s}}({\R}^q,E^j)\to {\Wcal}^{\tilde{s}-(\mu+N+1)}({\R}^q,{\tilde{H}}^j{\otimes}_HE^j)
$$
 for every $s\in\R$, and finally
$$
{\Op}_y(h)\otimes{\id}_E:H^{\y}\big({\R}^q,C^{\y}(X)\big)\to {\Wcal}^{\y}\big({\R}^q,{\Kcal}_{\Theta}^{\y,\gamma}(X^{\wedge})\big).
$$
\end{proof}

The case of variable discrete asymptotics will be prepared here by a number of specific observations. We saw  that the space \eqref{sing} is the image of $H^s({\R}^q)\hat{\otimes}_{\pi}C^{\y}(X)$ under the action of a pseudo-differential operators
$$
{\Op}_y(k)\hat{\otimes}_{\pi}{\id}_{C^{\y}(X)}:H^s({\R}^q)\hat{\otimes}_{\pi}C^{\y}(X)\to \K{\Wcal}^s({\R}^q,{\Ecal}_{\Pcal})_{\id}
$$
for symbols $k(\eta)\in S_{\clw}^0\big({\R}^q;\C,{\Kcal}^{\y,\gamma-n/2}({\R}_+)\big)$, $k(\eta):c\to k(\eta)c, c\in\C$, where $k(\eta):=\sum_{j=0}^J\sum_{k=0}^{m_j}c_{jk}[\eta]^{(n+1)/2}\omega(r[\eta])(r[\eta])^{-p_j}{\log}^k(r[\eta])$ for arbitrary constants $c,c_{jk}\in\C$, $0\le j\le m_j,$ $j=0,1,\ldots, J$.

Let us form the compact set $K:={\pi}_{\C}\Pcal=\{p_j\}_{j=0,1,\ldots,J}$ and choose any counter clockwise oriented (say, smooth) curve $C$ surrounding $K$ such that the winding number with respect to any $z\in K$ is equal to $1$. The function
$$
M_{r\to z}\big(\omega(r)\sum_{j=0}^J\sum_{k=0}^{m_j}c_{jk}(x)r^{-p_j}{\log}^kr\big)(z):=f(z)
$$
with $M$ being the weighted Mellin transform for the weight $\gamma-n/2$ is meromorphic with poles at the points $p_j$ of multiplicity $m_j+1$ and Laurent coefficients $(-1)^kk!c_{jk}(x)$. This comes from the identity
$$
M_{r\to z}\big(\omega(r)r^{-p}{\log}^kr\big)(z)=\frac{(-1)^kk!}{(z-p)^{k+1}}
$$
for any $p\in\C, k\in\N$, modulo an entire function. For any compact set $K\subset\C$ by ${\Acal}'(K)$ we denote the space of analytic functionals carried by $K$, see \cite[Vol. 1]{Horm5} or \cite[Section 2.3]{Kapa10}. The space ${\Acal}'(K)$ is a nuclear Fr\'{e}chet space. Given another Fr\'{e}chet space $E$ we set ${\Acal}'(K,E):={\Acal}'(K){\hat{\otimes}}_{\pi}E$. Now
\beq\label{f}
{\Acal}(\C)\ni h\to \la\zeta_{f,z},h\ra:=\int_{C}f(z)h(z)\dbar z
\eeq
is an analytic functional with carrier $K$, more precisely, $\zeta\in{\Acal}'\big(K, C^{\y}(X)\big)$. It is of finite order in the sense of a linear combination  of finite order derivatives of the Dirac measures at the points $p_j$. Inserting $h(z):=r^{-z}$ we just obtain
\beq\label{pz}
 \la\zeta_{f,z},r^{-z}\ra=\sum_{j=0}^J\sum_{k=0}^{m_j}c_{jk}(x)r^{-p_j}{\log}^kr,
\eeq
i.e., the singular functions are again reproduced as a linear superposition of $r^{-z}$ with the density $\zeta$.

The above-mentioned singular functions \eqref{sing} of constant discrete edge asy\-mptotics of type $\Pcal$ may be written in the form
$$
F^{-1}_{\eta\to y}\big\{[\eta]^{(n+1)/2}\omega(r[\eta])\la\hat{\zeta}(\eta)_z,(r[\eta])^{-z}\ra\big\}
$$
where $\hat{\zeta}(\eta)\in{\Acal'\big(K,{\hat{H}}^s\big({\R}^q_{\eta},C^{\y}(X)\big)\big)}$ is applied to $(r[\eta])^{-z}$; subscript $z$ indicates the pairing with respect to $z$. The form of $\hat{\zeta}(\eta)$ is subordinate to $\Pcal$ in the sense that $\la\hat{\zeta}(\eta)_z,r^{-z}\ra$ is a ${\hat{H}}^s\big({\R}^q_{\eta},C^{\y}(X)\big)$-valued meromorphic function with poles at the points $p_j\in\pi_{\C}{\Pcal}$ of multiplicity $m_j+1$. To have a notation, if $E$ is a Fr\'{e}chet space then a $\zeta\in{\Acal}'(K,E)$ is said to be subordinate to $\Pcal$ if $\la \zeta, r^{-z}\ra$ is meromorphic with such poles and multiplicities, determined by $\Pcal$. Let ${\Acal}'_{\Pcal}(K,E)$ denote the subspace of all $\zeta\in{\Acal}'(K,E)$ of that kind.

\section{Branching edge asymptotics} 

\subsection{Wedge spaces with branching edge asymptotics} 

The role of the present section is to deepen and complete  material from \cite{Schu67} on wedge space with variable branching edge asymptotics. To this end we first recall the notion of variable discrete asymptotic types.

Let ${\Ucal}(\Omega)$ for an open set $\Omega\subseteq{\R}^q$
denote the system of all open subsets $U\subset\Omega$ with compact
closure $\overline{U}\subset\Omega$.

\begin{defn}\label{v.ast}
A variable discrete asymptotic type $\Pcal$ over an open set
$\Omega\subseteq {\R}^q$ associated  with weight data $(\gamma,
\Theta), \Theta=(\vartheta, 0], -\y<\vartheta<0$, is a system of
sequences of pairs
\beq\label{v.pair}
\Pcal
(y)=\{\big(p_j(y),m_j(y)\big)\}_{j=0,1,\ldots, J(y)}
 \eeq
 for $J(y)\in\N$, $y\in \Omega,$ such that $\pi_{\C} \Pcal:=\{p_j(y)\}_{j=0,1,\ldots, J(y)}
 \subset\{(n+1)/2-\gamma+\vartheta<\reb \,z<(n+1)/2-\gamma\}$
  for all $y\in\Omega,$ and for every $b=(c,U),
  (n+1)/2-\gamma+\vartheta<c<(n+1)/2-\gamma$, $U\in{\Ucal
  }(\Omega)$,
   there are sets $\{U_i\}_{0\le i\le N}$, $\{K_i\}_{0\le i\le N}$,
    for some $N=N(b)$, where $U_i\in{\Ucal}(\Omega), 0\le i\le N,$
    form an open covering of $\overline{U}$, moreover,
 \beq\label{Ki}
 K_i\Subset\C, K_i\subset \{c-{\varepsilon}_i<\reb
    \,z<(n+1)/2-\gamma\} \quad\textup{for some} \quad {\varepsilon}_i>0,
 \eeq
such that
\beq\label{proj}
    \pi_{\C} \Pcal\cap\{c-{\varepsilon}_i<\reb \,z
    <(n+1)/2-\gamma\}\subset K_i \quad \textup{for all} \quad y\in U_i
\eeq
    and
    $$
    \sup_{y\in U_i}\sum_j\big(1+m_j(y)\big)<\y
     $$
     where the sum is taken over those $0\le j\le J(y)$ such that $p_j(y)\in K_i, i=0,1,
     \ldots, N.$
\end{defn}
We will say that a variable discrete asymptotic type $\Pcal$ satisfies the shadow condition if $(p(y),m(y))\in{\Pcal}(y)$ implies $(p(y)-l,m(y))\in{\Pcal}(y)$ for every $l\in\N$, such that $\reb \,p(y)-l>(n+1)/2-\gamma
+\vartheta$, for all $y\in\Omega$. Observe that such a condition is natural when we ask the spaces of functions $u$ with asymptotics \eqref{as} to be closed under multiplication by functions $\varphi\in C^{\y}(\overline{\R}_+),$ and then the Taylor asymptotics of $\varphi$ at $r=0$ contributes to the asymptotics of $\varphi u$. For any open
$\tilde{\Omega}\subseteq\Omega$ we define the restriction ${\Pcal}|_{\tilde{\Omega}}:=\{\big(p(y),m(y)\big)\in\Pcal : y\in\tilde{\Omega}\}$. We also define restrictions to
$A\subseteq\C$ by setting ${\ur}_{A}\Pcal :=\{\big(p(y),m(y)\big)\in\Pcal : p(y)\in A\}$.

In future if $K\subset\C$ is a compact set and we are  talking about a curve $C\subset\C\setminus K$ counter clockwise surrounding $K$ we tacitly assume that the winding number is $1$ with respect to every $z\in K$. It is well-known, that for every $K$ such a $C$ always exists in an $\varepsilon$-neighbourhood of $K$ for any $\varepsilon>0$.

Parallel to variable discrete  asymptotic types $\Pcal $ we consider families of analytic
functionals that are $y$-wise discrete and of finite order.
Typical families of that kind are generated by functions
 $f(y,z)\in C^{\y}\big(\Omega,{\Acal}(\C\setminus K)\big)$  that extend
 across $K$ for every $y\in\Omega$ to a meromorphic function in $z$, with finitely many poles
 $p_0(y), p_1(y), \ldots, p_{J}(y)\in K$ where $p_{j}(y)$ is of multiplicity $m_j(y)+1$. The corresponding system $\Pcal (y)$ of the form \eqref{v.pair}
is a variable discrete asymptotic type in the sense of Definition \ref{v.ast}.

More generally, if we have a family of meromorphic functions $f(y,z)$, paramet\-rized by $y\in\Omega$ we will say that $f$ is subordinate to \eqref{v.pair} if for every $y\in\Omega$ the system of poles is contained in ${\pi}_{\C}{\Pcal}(y)$ and and the multiplicities are $\le m_j(y)+1$. With such an $f(y,z)$ we can associate a family of analytic functionals as follows. We fix $b=(c,U)$ as in Definition \ref{v.ast} and choose a pair $(U_i,K_i)$ and a smooth curve
$C_i\subset\{c-{\varepsilon}_i<\reb\,z<(n+1)/2-\gamma\}$ counter clockwise surrounding $K_i$, and then we form $\delta_i(y)\in{\Acal}'(K_i)$ by
$$
\la\delta_i(y)_z,h\ra:=\int_Cf(y,z)h(z)\,\dbar z,
$$
$h\in{\Acal}(\C)$. The family $f$ is called smooth in $y\in\Omega$ if $\delta_i(y)\in C^{\y}\big(U_i, {\Acal}'(K_i)\big)$ for all $i=0,1,\ldots,N$, and if this is also the case for all $U\in{\Ucal}(\Omega)$.

In the following constructions it will be convenient to fix for any given $U\in{\Ucal}$ a system of $\varphi_i\in C^{\y}_0(U_i)$, $i=0,1,\ldots, N$, such that $\sum_{i=0}^N{\varphi}_i=1$ for all $y\in\overline{U}$. This yields a family
\beq\label{BE}
\delta_U(y):=\sum_{i=0}^N{\varphi}_i(y){\delta}_i(y)\in C^{\y}\big(U, {\Acal}'(K)\big)
\eeq
for $K:=\bigcup_{i=0}^NK_i$ which has the property that $M_{r\to z}\big(\omega(r)\la\delta_U(y)_w,r^{-w}\ra\big)$ is a family of meromorphic functions over $U$ equal to $f(y,z)|_{U}$ modulo a function in $C^{\y}\big(U, {\Acal}\big(c-\varepsilon<\reb\,z<(n+1)/2-\gamma\big)\big)$, $\varepsilon=\min\{\varepsilon_0,\varepsilon_1,\ldots,\varepsilon_N\}$.

Let us summarize these observations in the analogous case of $E$-valued meromorphic functions and $E$-valued analytic functionals as follows.

Given a Fr\'{e}chet space $E$ and a family of $E$-valued functions
$f(y,z)$ parametri\-zed by $y\in\Omega$ and meromorphic in
$(n+1)/2-\gamma+\vartheta<\reb \,z  <(n+1)/2-\gamma$, we say that
$f$ is subordinate to $\Pcal$ if every pole of $f(y, \cdot)$
belongs to a pair $(p(y),m(y))\in\Pcal$ where the
multiplicity is  less or equal $m(y)+1$.

Let $U\in\Ucal,
K\subseteq\C$, then $C^{\y}\big(U, \Acal(\C\setminus
K,E)\big)^{\bullet}$ will denote the subspace of all $f(y,z)\in
C^{\y}\big(U, \Acal(\C\setminus K,E)\big)$ that extend for every
$y\in U$ to a meromorphic function across $K$, again denoted by
$f(y,z)$, where poles and multiplicities minus $1$ form a
$\Pcal$ as in Definition \ref{v.ast}. If we specify
$\Pcal$ we also denote the space of such functions by $C^{\y}\big(\Omega,
{\Acal}_{\Pcal}(\C,E)\big)$.

If $f(y,z)$ is any family of meromorphic functions parametrized by
$y\in\Omega$ such that the pattern of poles together with
multiplicities minus $1$ is a $\Pcal$ as in Definition
\ref{v.ast} we may define smoothness in $y$ as follows. First we fix
any $y_0\in\Omega$ and a $b=(c, U)$ and sets $K_i, U_i, i=0,1,\ldots
N,$ as in Definition \ref{v.ast}. Choose compact smooth
 curves $C_i\subset\{c-{\varepsilon}_i<\reb \,z
  <(n+1)/2-\gamma\}$ counter clockwise surrounding $K_i$  and define ${\delta}_i(y)
 \in{\Acal}'(K_i, E)$ by
    $\la{\delta}_i(y)_z,h\ra:=\int_{C_i} f(y,z)h(z)\dbar z,$ $h\in{\Acal}(\C),
    y\in U_i$. Then $f$ is called smooth if ${\delta}_i\in C^{\y}\big(U_i,{\Acal}'(K_i,
   E)\big)$ for $i=0,1,\ldots,N$.

\begin{rem}\label{Me}
Consider the above-mentioned $f(y,z)$. Setting $f_i(y,z):=M_{r\to z}\omega(r)\newline \la{\delta}_i(y)_z, r^{-z}\ra$
with $M$ being the weighted Mellin transform for any weight $\beta$
such that ${\Gamma}_{1/2-\beta}\cap K_i=\emptyset$ we obtain an
element in $C^{\y}\big(U_i,{\Acal}(\C\setminus K_i, E)\big)$ subordinate to
${\mathcal P}|_{U_i}$. Clearly, in this case we have $f_i(y,z)\in
C^{\y}\big(U_i,{\Acal}(\C\setminus K_i, E)\big)$. Moreover, if
$\{{\varphi}_i\}_{i=0,1,\ldots,N}$ is a system ${\varphi}_j\in C_0^{\y}(U_j)$ such that $\sum_{j=0}^N{\varphi}_j\equiv 1$ over $\overline{U}\subset \bigcup_{i=0}^NU_i,$
then
$f_b(y,z):=\sum_{i=0}^N{\varphi}_i(y)f_i(y,z)$ satisfies the
relation $f|_{U}=f_b \quad \textup{mod} \quad C^{\y}\big(U,{\Acal}(c-\varepsilon<\reb \,z <(n+1)/2-\gamma, E)\big)$ for
$\varepsilon:=\min \{{\varepsilon}_0,{\varepsilon}_1,\ldots,
{\varepsilon}_N \}$.
\end{rem}

 Let us now recall from \cite{Schu67} the definition of weighted edge distributions of
variable discrete asymptotic type $\Pcal$, cf.~Definition
\ref{v.ast}.
\begin{defn}\label{sobast}
Let $\Omega\subseteq {\R}^q$ be open and let $\Pcal$ be a variable discrete asymptotic type, cf.~Definition \ref{v.ast}  associated with
the weight data $(\gamma, \Theta)$, $\Theta=(\vartheta,0]$ finite.
Then ${\Wcal}^s_{\loc}\big(\Omega, {\Kcal}^{s, \gamma}_{\Pcal}(X^{\wedge
})\big)$ for $s\in\R$ is defined to be the set of all
$u\in{\Wcal}^s_{\loc}\big(\Omega, {\Kcal}^{s, \gamma}(X^{\wedge })\big)$
such that for every $b:=(c,U)$  for any
$(n+1)/2-\gamma+\vartheta<c<(n+1)/2-\gamma$ and
$U\in\mathcal{U}(\Omega)$ there exists a compact set
$K_b\subset\{(n+1)/2-\gamma+\vartheta<\reb\,z<(n+1)/2-\gamma\}$ and a
function $\hat{f}_b(y,z,\eta)\in{C}^{\y}\big(U, {\Acal}(\C\setminus K_b,
E^s)\big)^{\bullet}$ for
\beq\label{Es}
E^s:={\hat{H}}^s\big({\R}^q_{\eta},C^{\y}(X)\big)
\eeq
subordinate to $\Pcal|_{U}$ and a corresponding
${\hat{\delta}}_b(y,\eta)\in{C}^{\y}\big(U, {\Acal}'(K_b,
E^s)\big)^{\bullet},$
\beq\label{De}
\la{\hat{\delta}}_b(y,\eta)_z,
h\ra=\int_{C_b}\hat{f}_b(y,z,\eta)h(z)\,\dbar z,\,\,h\in{\Acal(\C)},
\eeq
with $C_b$ counter clockwise surrounding $K_b$, such that
\beq\label{nsing}
u(r,x,y)-F^{-1}_{\eta\to
y}\{[\eta]^{(n+1)/2}\omega(r[\eta])\la{\hat{\delta}}_b(y,\eta)_z,
(r[\eta])^{-z}\ra\}\in{\Wcal}^s_{\loc}\big(U, {\Kcal}^{s,
\gamma+\beta}(X^{\wedge})\big)
\eeq
 for $\beta:={\beta}_0+\varepsilon $ for any $0<\varepsilon<\varepsilon(b),
 {\beta}_0:=(n+1)/2-\gamma-c.$ Moreover, we set
 $$
{\Wcal}^s_{\comp}\big(\Omega, {\Kcal}^{s, \gamma}_{\Pcal}(X^{\wedge
})\big):={\Wcal}^s_{\loc}\big(\Omega, {\Kcal}^{s, \gamma}_{\Pcal}(X^{\wedge
})\big)\cap{\Wcal}^s_{\comp}\big(\Omega, {\Kcal}^{s, \gamma}(X^{\wedge
})\big).
 $$
\end{defn}
For convenience, as a consequence of Definition \ref{sobast}, we
characterize the space ${\Wcal}^s_{\loc}\big(\Omega, {\Kcal}^{s,
\gamma}_{\Pcal}(X^{\wedge })\big)$ as the set of all
$u\in{\Wcal}^s_{\loc}(\Omega, {\Kcal}^{s, \gamma}\big(X^{\wedge })\big)$
such that for every $b=(c,U)$ the function $u|_{U}$ belongs to the
space
\beq\label{NS}
{\Wcal}^s_{\loc}\big(U, {\Kcal}^{s,
\gamma+\beta}(X^{\wedge })\big)+{\Wcal}^s_{b,\Pcal}(U)
\eeq
where
${\Wcal}^s_{b,\Pcal}(U):=\{F^{-1}_{\eta\to
y}({\kappa}_{[\eta]}\omega(r)\la{\hat{\delta}}_{b}(y,\eta)_z,r
^{-z}\ra)\}$, ${\hat{\delta}}_{b}(y,\eta)$ as in \eqref{De} for an
\newline${\hat{f}}_{b}(y,z,\eta)$ subordinate to
${\Pcal}_b:={\ur}_{K_b}({\Pcal}|_{U})$.

Definition \ref{sobast} expresses asymptotics of type $\Pcal$ in terms of pairs $U_i,K_i$ as in Definition \ref{v.ast}, i.e., localizations in $y\in\Omega$ and $z\in\C$. Therefore, for simplicity we focus on an open set $U\in{\Ucal}(\Omega)$ and a compact $K$ in the complex plane, $K\subset\{c-\varepsilon<\reb\,z<(n+1)/2-\gamma\}$ for some $\varepsilon>0$, such that ${\pi}_{\C}{\Pcal}\subset K$. This allows us to drop subscript $b$, i.e., we may write $K=K_b$, $\delta=\delta_b$,
\beq\label{Del}
{\hat{\delta}}(y,\eta)\in C^{\y}\big(U, {\Acal}'(K,E^s)\big)^{\bullet}.
\eeq
It is instructive to compare the notion of $y$-wise discrete asymptotics with continuous asymptotics where ${\hat{\delta}}(y,\eta)\in C^{\y}\big(U, {\Acal}'(K,E^s)\big)$.

Formally, the singular functions of continuous asymptotics are as before, namely, of the form
$$
F^{-1}_{\eta\to y}\{[\eta]^{(n+1)/2}\omega(r[\eta])\la\hat{\delta}(y,\eta)_z,(r[\eta])^{-z}\ra\}.
$$
In contrast to the latter explicit $y$-dependence of the analytic functionals there is also the case of constant continuous asymptotics  carried by the compact set $K$. In this case we can proceed in an analogous manner as in the constant discrete case, outlined in Subsection 1.2. When we fix the position of $K$ as above, i.e., $K\subset\{(n+1)/2-\gamma+\vartheta<\reb\,z<(n+1)/2-\gamma\}$,
then we have
$$
\omega(r)\la\zeta_z,r^{-z}\ra\subset{\Kcal}^{\y,\gamma}(X^{\wedge})
$$
for every $\zeta\in{\Acal}'\big(K,C^{\y}(X)\big)$, and
\beq\label{E}
{\Ecal}_K:=\{\omega(r)\la\zeta_z,r^{-z}\ra : \zeta\in{\Acal}'\big(K,C^{\y}(X)\big)\}
\eeq
is a continuous analogue of ${\Ecal}_{\Pcal}$ in \eqref{sing2}. Again we have ${\Kcal}_{\Theta}^{s,\gamma}(X^{\wedge})\bigcap{\Ecal}_K=\{0\}$ for any $s\in\R$, and analogously as \eqref{sub} we set
\beq\label{cros}
{\Kcal}_{\Ccal}^{s,\gamma}(X^{\wedge}):={\Kcal}_{\Theta}^{s,\gamma}(X^{\wedge})+{\Ecal}_K.
\eeq
The notation $\Ccal$ means that with $K$ we associate a corresponding continuous asymptotic type. The space ${\Ecal}_K$ is nuclear Fr\'{e}chet in a natural way via an isomorphism
\beq\label{iso}
{\Ecal}_K\cong{\Acal}'\big(K,C^{\y}(X)\big).
\eeq
Thus \eqref{cros} is Fr\'{e}chet in the topology of the direct sum. The group action
$\{{\kappa}_{\lambda}\}_{\lambda\in{\R}_+}$ defined by \eqref{group} is also defined on ${\Kcal}_{\Ccal}^{s,\gamma}(X^{\wedge})$ which allows us to define
$$
{\Wcal}^s\big({\R}^q,{\Kcal}_{\Ccal}^{s,\gamma}(X^{\wedge})\big):=
{\Wcal}^s\big({\R}^q,{\Kcal}_{\Theta}^{s,\gamma}(X^{\wedge})\big)+\K H^s({\R}^q,{\Ecal}_K).
$$
From \eqref{iso} it follows that
\beq\label{HE}
H^s({\R}^q_y,{\Ecal}_K)=\{\omega(r)\la\zeta(y)_z,r^{-z}\ra:\zeta\in{\Acal}'\big(K,H^s\big({\R}^q_y,C^{\y}(X)\big)\big)\}.
\eeq
Then
\beq\label{singK}
\begin{split}
\K H^s({\R}^q_y,{\Ecal}_K)&=\{F^{-1}_{\eta\to y}\kappa_{[\eta]}[\omega(r)F_{y'\to\eta}\la\zeta(y')_z,r^{-z}\ra]:\\ \zeta(y')&\in{\Acal}'\big(K,H^s\big({\R}^q_{y'},C^{\y}(X)\big)\big)\}.
\end{split}
\eeq
 Let us now make some general remarks about managing analytic functionals. If $E$ is a Fr\'{e}chet space and ${\Acal}'(K,E)$ the space of $E$-valued analytic functionals carried by the compact set $K\st\C$ we have
\beq\label{KC}
{\Acal}'(K,E)={\Acal}'(K^{\textup{c}},E)
\eeq
where $K^{\textup{c}}$ means the complement of the unbounded connected component of $\C\setminus K$, cf. \cite[Section 2.3]{Kapa10}.
Recall that the classical Cousin theorem also admits decompositions of the carrier, more precisely, if $K_1,K_2$ are compact sets in $\C$, then setting $K_1+K_2:=(K_1\cup K_2)^{\textup{c}}$ we have a non-direct sum of Fr\'{e}chet spaces
\beq\label{sum1}
{\Acal}'(K,E)={\Acal}'(K_1,E)+{\Acal}'(K_2,E),
\eeq
for any Fr\'{e}chet space $E$, cf. also \cite{Schu2}.

In the discussion so far we assumed that $K\bigcap\Gamma_{(n+1)/2-\gamma}=\emptyset$. However, in the edge calculus with continuous asymptotics also requires the case $K\bigcap\Gamma_{(n+1)/2-\gamma}\neq\emptyset$. Without loss of generality we may assume $K=K^{\textup{c}}$. Then \eqref{cros} is not direct and only $\{z\in K:\reb\,z>(n+1)/2-\gamma+\vartheta\}$ contributes to $\Ccal$. Writing $K$ as a sum $K=K_1+K_2$ for $K_1=\{z\in K:\reb\,z\le(n+1)/2-\gamma+\vartheta\}$, $K_2=\{z\in K:\reb\,z\ge(n+1)/2-\gamma+\vartheta\}$ we have a decomposition \eqref{sum1}. Therefore, every  $\zeta\in{\Acal}'(K,E)$ may be written as $\zeta=\zeta_1+\zeta_2$ for suitable $\zeta_i\in{\Acal}'(K_i,E), i=1,2$. This leads to a decomposition of the space \eqref{singK} as
$$
\K H^s(\R^q,{\Ecal}_K)=\K H^s(\R^q,{\Ecal}_{K_1})+\K H^s(\R^q,{\Ecal}_{K_2}).
$$
Clearly we have $\K H^s(\R^q,{\Ecal}_{K_1})\st {\Wcal}^s\big(\R^q,{\Kcal}_{\Theta}^{\y,\gamma}(X^{\wedge})\big)$, but also $K_2$ gives rise to a flat contribution, namely, from $K_0:=K_2\bigcap\Gamma_{(n+1)/2-\gamma+\vartheta}$. The notions and results that we are formulating here on continuous asymptotics have a natural modification for the case of arbitrary $K$. If necessary, we have to admit flat contributions.

\begin{prop}\label{aa}
For a compact set $K\st\{(n+1)/2-\gamma+\vartheta<\reb\,z<(n+1)/2-\gamma\}$ we have
\begin{equation*}
\begin{split}
\K H^s({\R}^q_y,{\Ecal}_K)=&\{\omega(r)F^{-1}_{\eta\to y}\kappa_{[\eta]}[F_{y'\to\eta}\la\zeta(y')_z,r^{-z}\ra]:\\ &\zeta(y')\in{\Acal}'\big(K,H^s\big({\R}^q_{y'},C^{\y}(X)\big)\big)\}
\end{split}
\end{equation*}

\nt ${\textup{mod}}\,{\Wcal}^s\big({\R}^q,{\Kcal}_{\Theta}^{\y,\gamma}(X^{\wedge})\big)$.
\end{prop}
\begin{proof}
Let us first drop $C^{\y}(X^{\wedge})$; this can be tensor-multiplied to the result in the final step, cf. the considerations in connection with \eqref{cont3}. For $\zeta$ we then  have
$$
\zeta\in {\Acal}'\big(K,H^s({\R}^q_{y'})\big)={\Acal}'(K){\hat{\otimes}}_{\pi}H^s({\R}^q_{y'}).
$$
We employ the fact that $\zeta$ can be written as a convergent sum
$\zeta=\sum_{j=0}^{\y}{\lambda}_j\zeta_jv_j$
for $\lambda_j\in\C, \sum_{j=0}^{\y}|\lambda_j|<\y, \zeta_j\in{\Acal}'(K), v_j\in H^s({\R}^q)$, tending to $0$ in the respective spaces, as $j\to\y$. Then, we form
$$
k_j(\eta):c\to \omega(r[\eta])[\eta]^{(n+1)/2}\la\zeta_{j,z},(r[\eta])^{-z}\ra c,
$$
$$
l_j(\eta):c\to \omega(r)[\eta]^{(n+1)/2}\la\zeta_{j,z},(r[\eta])^{-z}\ra c,
$$
$c\in \C$ and write
$$
d_j(\eta):=l_j(\eta)-k_j(\eta)=[\eta]^{(n+1)/2}\omega(r)\big(1-\omega(r[\eta])\big)\la\zeta_{j,z},(r[\eta])^{-z}\ra.
$$
We will show that
\beq\label{d}
d_j(\eta)\in S_{\clw}^0\big({\R}^q;\C,{\Kcal}^{\y,\beta}(\R_+)\big)
\eeq
for every $\beta\in\R$ and that $d_j(\eta)\to 0$ in that symbol spaces as $j\to\y$. This will give us
$$
{\Op}_y(d_j):H^s({\R}^q)\to {\Wcal}^s\big({\R}^q,{\Kcal}^{\y,\beta}(\R_+)\big).
$$
For fixed $v\in H^s({\R}^q)$ we can interpret
${\Op}_y(d_j)v={\Op}_y(l_j)v-{\Op}_y(k_j)v$ as
$$
F^{-1}_{\eta\to y}\big[ [\eta]^{(n+1)/2}\omega(r)\la\zeta_{j,z}{\hat{v}}(\eta),(r[\eta])^{-z}\ra\big]
-F^{-1}_{\eta\to y}\big[[\eta]^{(n+1)/2}\omega(r[\eta])\la\zeta_{j,z}{\hat{v}}(\eta),(r[\eta])^{-z}\ra\big],
$$
i.e., the difference between the respective singular functions for $\omega(r)$ and $\omega(r[\eta])$, respectively.

Let us now turn to \eqref{d} and set for the moment
$$
d(\eta)=[\eta]^{(n+1)/2}\omega(r)\big(1-\omega(r[\eta])\big)\la\zeta_z,(r[\eta])^{-z}\ra,
$$
i.e., we first drop subscript $j$. In order to show that
$d(\eta)\in S^0_{\clw}\big({\R}^q;\C,{\Kcal}^{\y,\beta}(\R_+)\big)$ we check the symbolic estimates
\beq\label{tt}
\|{\kappa}^{-1}_{[\eta]}D_{\eta}^{\delta}d(\eta)\|_{{\Lcal}(\C,{\Kcal}^{s,\beta}(\R_+))}=
\|{\kappa}^{-1}_{[\eta]}D_{\eta}^{\delta}d(\eta)\|_{{\Kcal}^{s,\beta}(\R_+)}\le c[\eta]^{-|\delta|},
\eeq
$\delta\in{\N}^q$, cf. the relation \eqref{est1}. It suffices to do that for every $s\in\N$, and we first consider the case $s=0$ and $\beta=0$.
Let $k^{\beta}(r)\in C_0^{\y}(\R_+)$ be any function that is strictly positive and
$k^{\beta}(r)=r^{\beta}$ for $0<r<c_0$, $k^{\beta}(r)=1$ for $r>c_1$, for some $0<c_0<c_1$.
Then ${\Kcal}^{s,\beta}(X^{\wedge})=k^{\beta}(r){\Kcal}^{s,0}(X^{\wedge})$.
In particular, by virtue of ${\Kcal}^{0,0}(\R_+)=L^2(\R_+)$ we have ${\Kcal}^{0,\beta}(\R_+)=k^{\beta}(r)L^2(\R_+)$ and
$$
\|f\|_{{\Kcal}^{0,\beta}(\R_+)}=\|k^{-\beta}f\|_{L^2(\R_+)}.
$$
In connection with \eqref{tt} we have to consider
$$
\|{\kappa}^{-1}_{[\eta]}d(\eta)\|_{{\Kcal}^{0,\beta}(\R_+)}=\|k^{-\beta}(r)\omega(r[\eta]^{-1})\big(1-\omega(r)\big)
\la\zeta_z,r^{-z}\ra\|_{L^2(\R_+)}.
$$
From the carrier of $\zeta$ we know that $\omega(r[\eta]^{-1})\la\zeta_z,r^{-z}\ra\in{\Kcal}^{\y,\gamma-n/2}(\R_+)$ for all $\zeta$; together with the factor $k^{-\beta}(r)\big(1-\omega(r)\big)$ we get $k^{-\beta}(r)\big(1-\omega(r)\big)\omega(r[\eta]^{-1})\la\zeta_z,r^{-z}\ra\in L^2(\R_+)$.
It follows that
$\|{\kappa}^{-1}_{[\eta]}d(\eta)\|_{{\Kcal}^{0,\beta}(\R_+)}\le c$
for all $\eta\in{\R}^q$. For the $\eta$-derivatives we obtain \eqref{tt} in general. Let us check, for instance, the case $\delta=(1,0,\ldots, 0)$, i.e., $D_{\eta}^{\delta}=-i\partial_{{\eta}_1}$. In this case we have
\begin{equation*}
\begin{split}
&\partial_{{\eta}_1}d(\eta)=(\partial_{{\eta}_1}[\eta]^{(n+1)/2})\omega(r)\big(1-\omega(r[\eta])\big)\la\zeta_z,(r[\eta])^{-z}\ra-
r[\eta]^{(n+1)/2}(\partial_{{\eta}_1}[\eta])
\\
&
\omega(r){\omega}'(r[\eta])\la\zeta_z,(r[\eta])^{-z}\ra-
[\eta]^{(n+1)/2}\omega(r)\big(1-\omega(r[\eta])\big)\la\zeta_z,z[\eta]^{-1}(\partial_{{\eta}_1}[\eta])(r[\eta])^{-z}\ra.
\end{split}
\end{equation*}
This gives us the desired estimate with $[\eta]^{-1}$ on the right. The general case may easily be treated in a similar manner. Now an elementary consideration shows that the constants $c=c(\zeta)$ in the symbolic estimates \eqref{tt} tend to $0$ as $\zeta\to 0$ in ${\Acal}'(K)$. Moreover, we can easily treat the case ${\Kcal}^{s,\beta}(\R_+)$ rather than ${\Kcal}^{0,\beta}(\R_+), s\in\N$. This implies the asserted estimates for all $s\in\R$. In other words, as claimed above, $d_j(\eta)=l_j(\eta)-k_j(\eta)$ tends to $0$ in $S_{\clw}^0\big({\R}^q;\C,{\Kcal}^{\y,\beta}(\R_+)\big)$ as $j\to\y$.

Now we characterize the difference between the singular terms defined with $\omega(r)$ and $\omega(r[\eta])$, respectively. It is equal to

\begin{equation*}
\begin{split}
&F^{-1}_{\eta\to y}\omega(r)\big[\kappa_{[\eta]}F_{y'\to\eta}\la\zeta(y')_z,r^{-z}\ra\big]-
F^{-1}_{\eta\to y}\omega(r[\eta])\big[\kappa_{[\eta]}F_{y'\to\eta}\la\zeta(y')_z,r^{-z}\ra\big]\\
=&F^{-1}_{\eta\to y}\omega(r)\big(1-\omega(r[\eta]\big)\big[\kappa_{[\eta]}F_{y'\to\eta}\la\zeta(y')_z,r^{-z}\ra\big]\\
=&F^{-1}_{\eta\to y}\omega(r)\big(1-\omega(r[\eta])\big)\kappa_{[\eta]}F_{y'\to\eta}\big\la\sum_{j=0}^{\y}{\lambda}_j\zeta_{j,z}v_j(y'),r^{-z}\big\ra\\
=&\sum_{j=0}^{\y}{\lambda}_jF^{-1}_{\eta\to y}\omega(r)\big(1-\omega(r[\eta])\big)\la\zeta_{j,z},(r[\eta])^{-z}\ra{\hat{v}}_j(\eta)=\sum_{j=0}^{\y}\lambda_j{\Op}_y(d_j)v_j.
\end{split}
\end{equation*}
This sum converges in ${\Wcal}^s\big({\R}^q,{\Kcal}^{\y,\beta}(\R_+)\big)$.

In fact, for every $t\ge 0$ we have
\begin{equation}\label{Su}
\begin{split}
\big\|\sum_{j=0}^{\y}\lambda_j{\Op}_y(d_j)v_j&\big\|_{{\Wcal}^s({\R}^q,{\Kcal}^{t,\beta}(\R_+))}\le \sum_{j=0}^{\y}|\lambda_j|\,\|{\Op}_y(d_j)v_j\|_{{\Wcal}^s({\R}^q,{\Kcal}^{t,\beta}(\R_+))}\\
\le\sum_{j=0}^{\y}|\lambda_j|\,&\|{\Op}_y(d_j)\|_{{\Lcal}(H^s({\R}^q),{\Wcal}^{s}({\R}^q,{\Kcal}^{t,\beta}(\R_+)))}\|v_j\|_{H^s({\R}^q)}.
\end{split}
\end{equation}
By virtue of \eqref{NO} we have
$$
\|{\Op}_y(d_j)\|_{{\Lcal}(H^s({\R}^q),{\Wcal}^{s}({\R}^q,{\Kcal}^{t,\beta}(\R_+)))}\to 0
$$
as $j\to\y$. Then $v_j\to 0$ in $H^s({\R}^q)$ as $j\to\y$ shows the convergence of the right hand side of \eqref{Su} for every $t\ge 0$, and hence it follows that
$$
\sum_{j=0}^{\y}\lambda_j{\Op}_y(d_j)v_j\in {\Wcal}^s\big({\R}^q,{\Kcal}^{\y,\beta}(\R_+)\big).
$$
So far we considered the case without $C^{\y}(X)$. However, as illustrated at the beginning, a tensor product argument gives us the result in general.
\end{proof}

Let us finally discuss to what extent the singular functions of variable branching or continuous edge asymptotics depend on the specific choice of the function $\eta\to[\eta]$. The other ``non-classical'' ingredient, namely, the cut-off function $\omega$ has been considered before. After  Proposition \ref{aa} it is clear that changing $\omega$ only causes a flat remainder. If we replace $[\eta]$ by an $[\eta]_1$ of analogous properties we obtain smoothing remainders with asymptotics. More precisely we have the following behaviour.
\begin{rem}\label{diff}
For any $\zeta\in{\Acal}'\big(K,H^s\big({\R}^q_{y'},C^{\y}(X)\big)\big), K\st\{\reb\,z<(n+1)/2-\gamma\}$, the difference
\beq\label{cp}
\omega(r)F^{-1}_{\eta\to y}\kappa_{[\eta]}\la{\hat{\zeta}}_z,r^{-z}\ra-\omega(r)F^{-1}_{\eta\to y}{\kappa}_{[\eta]_1}\la{\hat{\zeta}}_z,r^{-z}\ra
\eeq
belongs to $\in{\Wcal}^{\y}\big(\R^q,{\Kcal}_{\Ccal}^{\y,\gamma}(X^{\wedge})\big)$,
cf. the notation \eqref{cros}.
\end{rem}

In fact, \eqref{cp} has compact support in $\eta\in\R^q$. We have
\beq\label{dd}
\begin{split}
[\eta]^{(n+1)/2}&\la\hat{\zeta}_z,(r[\eta])^{-z}\ra-[\eta]_1^{(n+1)/2}\la\hat{\zeta}_z,(r[\eta]_1)^{-z}\ra
=[\eta]^{(n+1)/2}\frac{[\eta]^{(n+1)/2}-[\eta]_1^{(n+1)/2}}{[\eta]^{(n+1)/2}}\\
&\la\hat{\zeta}_z,(r[\eta])^{-z}\ra+[\eta]^{(n+1)/2}
\Big(\frac{[\eta]_1}{[\eta]}\Big)^{(n+1)/2}\la\hat{\zeta}_z,(r[\eta])^{-z}\frac{[\eta]^{-z}-
{[\eta]_1}^{-z}}{[\eta]^{-z}}\ra.
\end{split}
\eeq
For the first summand we employ that
$$
\frac{[\eta]^{(n+1)/2}-[\eta]_1^{(n+1)/2}}{[\eta]^{(n+1)/2}}\hat{\zeta}
=:\hat{\nu}\in{\Acal}'\big(K,{\hat{H}}^{\y}\big({\R}^q_{\eta},C^{\y}(X)\big)\big)
$$
since $[\eta]=[\eta]_1$, for large $|\eta|$. Moreover, we have
$$
\Big(\frac{[\eta]_1}{[\eta]}\Big)^{(n+1)/2}\frac{[\eta]^{-z}-{[\eta]_1}^{-z}}{[\eta]^{-z}}\hat{\zeta}
=:\hat{\sigma}\in{\Acal}'\big(K,{\hat{H}}^{\y}\big({\R}^q_{\eta},C^{\y}(X)\big)\big).
$$
Thus \eqref{dd} is equal to $[\eta]^{(n+1)/2}\la(\hat{\nu}+\hat{\sigma})_z,(r[\eta])^{-z}\ra$ and hence \eqref{cp} is equal to
$$
F^{-1}_{\eta\to y}[\eta]^{(n+1)/2}\la(\hat{\nu}+\hat{\sigma})_z,(r[\eta])^{-z}\ra
$$
which belongs to ${\Wcal}^{\y}\big(\R^q,{\Kcal}_{\Ccal}^{\y,\gamma}(X^{\wedge})\big)$.

\subsection{The Sobolev regularity of coefficients in branching edge asymptotics} 

Our next objective is to consider singular functions of continuous edge asymptotics, described in terms of smooth functions on $y\in\Omega$ with compact support with values in  ${\Acal}'\big(K,{\hat{H}}^s\big({\R}^q_{\eta},C^{\y}(X)\big)\big)$. We show that those functions may be represented by functionals without dependence on $y$. A similar result has been formulated in \cite[Proposition 3.1.35]{Schu20},  but here we give an alternative proof, and we obtain more information.  For convenience we start with Schwartz functions in $y\in{\R}^q$ which covers the case of functions with compact support in $y\in\Omega$. In addition we always write $\omega(r)$ rather than $\omega(r[\eta])$ which is admitted for similar reasons as in Proposition \ref{aa}, modulo flat remainders.

\begin{thm}\label{reg}
Let  ${\hat{\zeta}}(y,\eta)\in \Scal\big({\R}^q,{\Acal}'\big(K,{\hat{H}}^s\big({\R}^q_{\eta},C^{\y}(X)\big)\big)\big)$, $K\st\{(n+1)/2-\gamma+\vartheta<\reb\,z<(n+1)/2-\gamma\}$ compact, and form
\beq\label{ff}
f(r,y):=F^{-1}_{\eta\to y}\{[\eta]^{(n+1)/2}\omega(r)\la{\hat{\zeta}}(y,\eta)_z,(r[\eta])^{-z}\ra\}
\eeq
$($the dependence on $x\in X$ is dropped in the notation$)$. Then there is a unique ${\hat{\chi}}\in {\Acal}'\big(K,{\hat{H}}^s\big({\R}^q_{\eta},C^{\y}(X)\big)\big)$ such that
\beq\label{Uni}
f(r,y):=F^{-1}_{\eta\to y}\{[\eta]^{(n+1)/2}\omega(r)\la{\hat{\chi}}(\eta)_z,(r[\eta])^{-z}\ra\},
\eeq
and the correspondence $\hat{\zeta}\to\hat{\chi}$ defines an operator
\beq\label{abb}
B:{\Scal}\big({\R}^q, {\Acal}'\big(K,{\hat{H}}^s\big({\R}^q_{\eta},C^{\y}(X)\big)\big)\big)\to{\Acal}'\big(K,{\hat{H}}^s\big({\R}^q_{\eta},C^{\y}(X)\big)\big).
\eeq
\end{thm}
\begin{proof}
We employ some background on the pseudo-differential calculus with opera\-tor-valued symbols of the kind $S_{\cl}^{\mu}(\Omega\times{\R}^q;H,\tilde{H})$ with twisted symbolic estimates \eqref{ests}.
 In our case we set $\Omega={\R}^q$ and look at the subspace
${\Scal}\big({\R}^q_y,S_{\cl}^{\mu}({\R}^q_{\eta};H,\tilde{H})\big)$.
Given an $a_{\textup{L}}(y,\eta)$ in that space we have (by notation) the situation of a left symbol in the calculus of pseudo-differential operators ${\Op}_y(a_{\textup{L}})$, cf. the expression \eqref{Osc} where the respective amplitude function is a double symbol. It will be necessary to generate right symbols $a_{\textup{R}}(y',\eta)$ such that
\beq\label{Id}
 {\Op}_y(a_{\textup{L}})={\Op}_y(a_{\textup{R}}).
\eeq
A modification of the Kumano-go's global (in $\R^q$) pseudo-differential calculus is that
$a_{\textup{L}}\to a_{\textup{R}}$ with \eqref{Id} defines continuous operator
$$
{\Scal}\big({\R}^q_y,S_{\cl}^{\mu}({\R}^q_{\eta};H,\tilde{H})\big)\to {\Scal}\big({\R}^q_{y'},S_{\cl}^{\mu}({\R}^q_{\eta};H,\tilde{H})\big).
$$
Using an expansion for $a_{\textup{R}}$ with remainder we have, in particular,
\beq\label{right}
a_{\textup{R}}(y',\eta)=a_{\textup{L}}(y', \eta)+r_{\textup{R}}(y',\eta)
\eeq
for
\beq\label{rest}
r_{\textup{R}}(y',\eta)=-\sum_{|\alpha|=1}\int_0^1\!\!\!\int\!\!\!\!\int e^{-ix\xi}(D_y^{\alpha}\partial_{\eta}^{\alpha}a)(y'+x,\eta-t\xi)\,dx\dbar \xi dt.
\eeq
Here $\partial_{\eta}^{\alpha}=\partial_{\eta_1}^{\alpha_1}\ldots\partial_{\eta_q}^{\alpha^q}$  and $D_y^{\alpha}=(-i)^{|\alpha|}\partial_{y}^{\alpha}$ for $\alpha=(\alpha_1,\ldots,\alpha_q),|\alpha|=\alpha_1+\ldots+\alpha_q$.
The map $a_{\textup{L}}(y,\eta)\to r_{\textup{R}}(y',\eta)$ defines a continuous operator
\beq\label{cc}
{\Scal}\big({\R}^q_y,S_{\cl}^{\mu}({\R}^q_{\eta};H,\tilde{H})\big)\to {\Scal}\big({\R}^q_{y'},S_{\cl}^{\mu-1}({\R}^q_{\eta};H,\tilde{H})\big).
\eeq
In our concrete situation similarly as before we first look at the case without $C^{\y}(X)$; then a tensor product consideration gives us the result in general. We express $\hat{\zeta}(y,\eta)\in {\Scal}\big(\R^q,{\Acal}'\big(K,{\hat{H}}^s(\R_{\eta}^q)\big)\big)$ as an expansion
$$
\hat{\zeta}(y,\eta)=\sum_{j=0}^{\y}\lambda_j\zeta_j\varphi_j(y){\hat{v}}_j(\eta)
$$
for $\lambda_j\in\C, \sum_{j=0}^{\y}|\lambda_j|<\y,\varphi_j\in{\Scal}({\R}^q_{y}), v_j\in H^s({\R}^q_{y'})$, tending to zero in the respective spaces. This allows us to write the function \eqref{ff} in the form
$$
f(r,y)=\sum_{j=0}^{\y}\lambda_jF^{-1}_{\eta\to y}\{[\eta]^{(n+1)/2}\omega(r)\varphi_j(y)\la\zeta_{j,z},(r[\eta])^{-z}\ra{\hat{v}}_j(\eta)\}=\sum_{j=0}^{\y}\lambda_j{\Op}_y(k_j)v_j
$$
where $k_j(y,\eta)\in {\Scal}\big(\R_y^q,S^{\mu}({\R}^q_{\eta};\C,{\tilde{H}}_l)\big)$
is defined by
$$
k_j(y,\eta):c\to [\eta]^{(n+1)/2}\omega(r)\varphi_j(y)\la\zeta_{j,z},(r[\eta])^{-z}\ra c,
$$
and ${\tilde{H}}_l, l\in\N$, is a scale of Hilbert spaces with $\kappa$-action such that
$$
{\Kcal}_{\Ccal}^{\y,\gamma-n/2}(\R_+)=\lim_{\overleftarrow{l\in\N}}\tilde{H}_l,
$$
cf. equation \eqref{cros}. In other words we apply the above general relations on symbols to the case $H:=\C$ with the trivial group action and $\tilde{H}={\tilde{H}}_l$ endowed with $\kappa$, for every fixed $l$. Writing for the moment $k_j(y,\eta)=k_{j,\textup{L}}(y,\eta)$ we obtain a right symbol $k_{j,\textup{R}}(y',\eta)$ which is of the form
$$
k_{j,\textup{R}}(y',\eta)c=[\eta]^{(n+1)/2}\omega(r)\varphi_j(y')\la\zeta_{j,z},(r[\eta])^{-z}\ra c+r_{j,\textup{R}}(y',\eta)c,
$$
where $r_{j,\textup{R}}$ is obtained from \eqref{rest} for $a_{\textup{L}}=k_{j,\textup{L}}$. Let us consider for the moment the case $q=1$, and then write $y=y_1, \eta=\eta_1$. The general case is completely analogous. Later on in the function and symbol spaces we tacitly return again to $\R^q$ rather than $\R^1$. Then the remainder expression takes the form
\begin{equation*}
\begin{split}
r_{j,\textup{R}}(y',\eta)=-\int_0^1\!\!\!\int\!\!\!\!\int e^{-ix\xi}r^{-(n+1)/2}\omega(r)&(D_y\varphi_j)(y'+x)\\ &\la\zeta_{j,z},\big(\partial_{\eta}((r[\eta])^{-z+(n+1)/2})\big)|_{\eta-t\xi}\ra\,dx\dbar \xi dt.
\end{split}
\end{equation*}

\nt We now apply an element of Kumano-go's calculus for scalar symbols and observe that
$$
d_j(z,y',\eta)=\int_0^1\!\!\!\int\!\!\!\!\int e^{-ix\xi}(D_y\varphi_j)(y'+x)
\big(\partial_{\eta}([\eta]^{-z+(n+1)/2})\big)|_{\eta-t\xi}\,dx\dbar \xi dt
$$
belongs to ${\Scal}\big(\R_{y'}, S^{-\reb\,z +(n+1)/2-1}(\R_{\eta})\big)$ for every fixed $z$. In addition $d_j(z,y',\eta)$ is an entire function in $z$. This gives us
\beq\label{DD}
\begin{split}
r_{j,\textup{R}}(y',\eta)=&-r^{-(n+1)/2}\omega(r)\la\zeta_{j,z},r^{-z+(n+1)/2}d_j(z,y',\eta)\ra=\\
&-r^{-(n+1)/2+1}\omega(r)\la{\hat{\delta}}_{j,z}(y',\eta),(r[\eta])^{-z-1+(n+1)/2}\ra
\end{split}
\eeq
for ${\hat{\delta}}_j(y',\eta):=\zeta_jd_j(z,y',\eta)/[\eta]^{-z-1+(n+1)/2}$.
 We now employ the fact that the pseudo-differential action with a right symbol $b(y',\eta)$, say, in the scalar case $b(y',\eta)\in {\Scal}\big({\R}^q_{y'},S^{\nu}({\R}_{\eta}^q)\big)$ for some $\nu$, operating on $v\in H^s(\R^q)$ has the form
$$
{\Op}_y(b)v=\int e^{iy\eta}\left\{\int e^{-iy'\eta}b(y',\eta)v(y')\,dy'\right\}\dbar\eta.
$$
In order to analyze the expression we may apply a tensor product expansion
$$
b(y', \eta)=\sum_{l=0}^{\y}{\gamma}_l{\psi}_l(y')b_l(\eta)
$$
with $\sum_{l=0}^{\y}|{\gamma}_l|<\y, {\psi}_l\in{\Scal}({\R}^q), b_l\in S^{\nu}(\R^q)$, tending to zero in the considered spaces when $l\to\y$. Then
\begin{equation*}
\begin{split}
{\Op}_y(b)v=&\int e^{iy\eta}
\left\{\int e^{-iy'\eta}\sum_{l=0}^{\y}{\gamma}_l{\psi}_l(y')b_l(\eta)v(y')\,dy'\right\}\dbar\eta=\\
&\int e^{iy\eta}\sum_{l=0}^{\y}{\gamma}_lb_l(\eta){\widehat{\psi_lv}}(\eta)\,dy'\dbar\eta.
\end{split}
\end{equation*}
We have $\psi_lv\in H^s({\R}^q_{y'}), \psi_lv\to 0$ in $H^s({\R}^q_{y'})$, and we obtain altogether a sum
$$
{\Op}_y(b)v=\sum_{l=0}^{\y}{\gamma}_l{\Op}_y(b_l)(\psi_l v),
$$
convergent in $H^{s-\nu}({\R}^q)$. This  consideration may be modified for the present case.

\nt Let us write  \eqref{DD} as
\beq\label{rest1}
r_{j,\textup{R}}(y',\eta)=-[\eta]^{(n+1)/2-1}\omega(r)\la{\hat{\delta}}_{j,z}(y',\eta),(r[\eta])^{-z}\ra.
\eeq
We have
$$
{\hat{\delta}}_j(y',\eta)=\sum_{l=0}^{\y}{\gamma}_l\psi_l(y')b_{jl}(\eta),
$$
for ${\hat{\delta}}_j(y',\eta)\in {\Acal}'\big(K,{\Scal}\big({\R}^q_{y'},S^0_{\clw}({\R}^q_{\eta})\big)\big)$, where ${\hat{b}}_{jl}(\eta)\in{\Acal}'\big(K,S^0_{\clw}({\R}^q_{\eta})\big)$. We employ the fact that the pairing $S^0_{\clw}({\R}^q_{\eta})\times{\hat{H}}^s({\R}^q_{\eta})\to{\hat{H}}^s({\R}^q_{\eta})$
gives rise to a bilinear map
$$
\big({\id}_{{\Acal}'(K)}\otimes S^0_{\clw}({\R}^q_{\eta})\big)\times{\hat{H}}^s({\R}^q_{\eta})\to{\Acal}'(K){\hat{\otimes}}_{\pi}{\hat{H}}^s({\R}^q_{\eta}).
$$
It follows that
$$
r_{j,\textup{R}}(y',\eta)=-[\eta]^{(n+1)/2-1}\omega(r)\big\la\sum_{l=0}^{\y}{\gamma}_l\psi_l(y')b_{jl}(\eta),(r[\eta])^{-z}\big\ra
$$
and
$$
{\Op}_y(r_{j,\textup{R}})v_j(y)=F^{-1}_{\eta\to y}\big\{-[\eta]^{(n+1)/2-1}\omega(r)
\sum_{l=0}^{\y}{\gamma}_l\big\la b_{jl,z}(\eta),(r[\eta])^{-z}\big\ra\widehat{\psi_lv_j}(\eta)\big\}.
$$
For
\beq\label{chi}
{\hat{\chi}}_{j,\textup{rest}}(\eta):=\sum_{l=0}^{\y}{\gamma}_lb_{jl}(\eta)\widehat{\psi_lv_j}(\eta)\in{\Acal}'\big(K,{\hat{H}}^s(\R^q)\big)
\eeq
it follows that
$$
{\Op}_y(r_{j,\textup{R}})v_j(y)=F^{-1}_{\eta\to y}\{-[\eta]^{(n+1)/2-1}\omega(r)
\la{\hat{\chi}}_{j,z} (\eta),(r[\eta])^{-z}\ra\}.
$$
Returning to \eqref{right} from \eqref{rest1} we obtain
$$
r_{\textup{R}}(y',\eta)=-[\eta]^{(n+1)/2-1}\omega(r)\sum_{j=0}^{\y}{\lambda}_j\big\la{\hat{\delta}}_{j,z}(y',\eta),(r[\eta])^{-z}\big\ra
$$
and
$$
F^{-1}_{\eta\to y}\big(F_{y'\to\eta}r_{j,\textup{R}})(y',\eta)\big)=-F^{-1}_{\eta\to y}\big\{[\eta]^{(n+1)/2-1}\omega(r)\sum_{j=0}^{\y}{\lambda}_j\big\la{\hat{\chi}}_{j,z}(\eta),(r[\eta])^{-z}\big\ra\big\}.
$$
By notation we have
$k_{\textup{L}}(y,\eta)=\sum_{j=0}^{\y}{\lambda}_j k_{j,\textup{L}}(y,\eta)$
where $k_{j,\textup{L}}(y,\eta)\to 0$ in ${\Scal}\big({\R}^q_y,\newline S^0(\R^q;\C,{\tilde{H}}_l)\big)$ and then  $k_{j,\textup{R}}(y',\eta)\to 0$ in ${\Scal}\big({\R}^q_y,S^0(\R^q;\C,{\tilde{H}}_l)\big)$ and  $r_{j,\textup{R}}(y',\eta)\to 0$ in ${\Scal}\big({\R}^q_y,S^{-1}(\R^q;\C,{\tilde{H}}_l)\big)$ as $j\to\y$. This implies that
$k_{\textup{R}}(y',\eta)=\sum_{j=0}^{\y}{\lambda}_j k_{j,\textup{R}}(y',\eta)$.
 We obtain that $\hat{\chi}_{j,\textup{rest}}(\eta)\to 0$ in ${\Acal}'\big(K,{\hat{H}}^s(\R^q)\big)$ as $j\to\y$, cf. \eqref{chi},  hence it follows an element
$$
{\hat{\chi}}_{\textup{rest}}(\eta):=\sum_{j=0}^{\y}{\lambda}_j{\hat{\chi}}_{j,\textup{rest}}(\eta)\in{\Acal}'\big(K,{\hat{H}}^s(\R^q)\big).
$$
In a similar (simpler) manner we can treat the term $a_{\textup{L}}(y',\eta)$, cf. \eqref{right}, which gives us a ${\hat{\chi}}_{\textup{main}}\in{\Acal}'\big(K,{\hat{H}}^s(\R^q_{\eta})\big)$, and it follows altogether
\begin{equation*}
\begin{split}
f(r,y)=&F^{-1}_{\eta\to y}\{[\eta]^{(n+1)/2}\omega(r)\la{\hat{\chi}}_{\textup{main}}
(\eta)_z,(r[\eta])^{-z}\ra\}\\
-&F^{-1}_{\eta\to y}\{[\eta]^{(n+1)/2-1}\omega(r)\la{\hat{\chi}}_{\textup{rest}}
(\eta)_z,(r[\eta])^{-z}\ra\}.
\end{split}
\end{equation*}
Note that $[\eta]^{-1}{\hat{\chi}}_{\textup{rest}}\in{\Acal}'\big(K,{\hat{H}}^{s+1}(\R^q_{\eta})\big)
\hookrightarrow{\Acal}'\big(K,{\hat{H}}^s(\R^q_{\eta})\big)$.
 Analogous considerations apply for the $C^{\y}(X)$-valued case. We thus obtain the claimed representation \eqref{Uni}
where
$$
{\hat{\chi}}(\eta):={\hat{\chi}}_{\textup{main}}(\eta)-[\eta]^{-1}{\hat{\chi}}_{\textup{rest}}(\eta)\in
{\Acal}'\big(K,H^s\big(\R^q,C^{\y}(X)\big)\big).
$$
Let us now prove the uniqueness of $\hat{\chi}$ in the formula \eqref{Uni}. Without loss of generality we assume $K=K^{\textup{c}}$, cf. the relation \eqref{KC}. We have an isomorphism
$$
{\Acal}'(K,E)\cong\{\omega(r)\la\chi_z,r^{-z}\ra:\chi\in{\Acal}'(K,E)\}
$$
where on the right hand side we talk about functions in $C^{\y}(\R_+,E)$, and $\omega$ is a fixed cut-off function. Clearly we know much more about such functions; they belong to ${\Kcal}^{\y,\gamma}(\R_+,E)$ where $\gamma\in\R$ is any real such that $K\st\{\reb\,z<1/2-\gamma\}$. The notation ${\Kcal}^{\y,\gamma}(\R_+,E)$ is an $E$-valued generalization of the above-mentioned ${\Kcal}^{\y,\gamma}(\R_+)$. Up to a translation in the complex plane we may assume $\gamma=0$. Then the Mellin transform
$$
M_{r\to w}\big(\omega(r)\la\chi_z,r^{-z}\ra\big)=:m(w)
$$
gives us an element in $L^2(\Gamma_{1/2},E)$ which is holomorphic in $\C\setminus{K^{\textup{c}}}$, and we can recover $\chi$ by forming
$$
\chi:h\to\int_Cm(w)h(w)\,\dbar w,\quad h\in{\Acal}(\C)
$$
for any $C$ counter clockwise surrounding $K$.

 The multiplication of a $\chi\in{\Acal}'(K,E)$ by $g\in{\Acal}(\C)$, defined by $\la\chi,h\ra:=\la\chi,gh\ra$ gives us again an element in ${\Acal}'(K,E)$.
 Now looking at the expression \eqref{Uni} it suffices to recover
$$
\hat{\vartheta}(\eta):=[\eta]^{(n+1)/2}\hat{\chi}(\eta)\in{\Acal}'\big(K,{\hat{H}}^{s-(n+1)/2}\big({\R}^q_{\eta},C^{\y}(X)\big)\big)
$$
from
$$
F_{y\to\eta}(f)(r,\eta)=\omega(r)\la\hat{\vartheta}(\eta),(r[\eta])^{-z}\ra=\omega(r)\la[\eta]^{-z}\hat{\vartheta}(\eta),r^{-z}\ra
$$
the Mellin transform of which belongs to ${\Acal}\big(\C\setminus K,{\hat{H}}^{s-(n+1)/2}\big({\R}_{\eta}^q,C^{\y}(X)\big)\big)$ where
$$
[\eta]^{-w}\hat{\vartheta}(\eta):h\to\int_CM_{r\to w}\big(\omega(r)\la[\eta]^{-z}\hat{\vartheta}(\eta),r^{-z}\ra\big)h(w)\,\dbar w.
$$
Thus we find $[\eta]^{-w}\hat{\vartheta}(\eta)$ and hence $\hat{\vartheta}(\eta)$ itself by composing the result with the entire function $[\eta]^w$. In other words $\hat{\chi}$ in the formula \eqref{Uni} is unique.
\end{proof}

Let us now discuss the Sobolev regularity of coefficients in the singular functions of edge asymptotics. In order to illustrate what we mean we first look at constant discrete asymptotics of type $\Pcal$. According to Proposition \ref{aa} the singular functions are finite linear combinations of expressions
$$
\omega(r)F^{-1}_{\eta\to y}\{[\eta]^{(n+1)/2}(r[\eta])^{-p}\,{\log}^k(r[\eta]){\hat{v}}_{p,k}(\eta,x)\},
$$
for ${\hat{v}}_{p,k}(\eta,x)\in {\hat{H}}^s\big(\R^q,C^{\y}(X)\big), p\in{\pi}_{\C}{\Pcal}$, and some $k\in\N$, cf. the formulas \eqref{group}, \eqref{si} and \eqref{pz}.
 The $\eta$-dependence lies in
\beq\label{coef}
[\eta]^{(n+1)/2-p}\,{\log}^l[\eta]{\hat{v}}_{p,k}(\eta,x)=:{\hat{w}}_{p,k}(\eta,x)
\eeq
for some $0\le l\le k$, i.e.,
\beq\label{regass}
{w}_{p,k}(y,x)\in H^{s+\reb\,p-\varepsilon-(n+1)/2}\big({\R}^q_{y},C^{\y}(X)\big),
\eeq
for any $\varepsilon>0$.
The case of constant continuous asymptotics can be interpreted in terms of Sobolev regularity as well. Here in the representation as in Proposition \ref{aa} the analytic functional $\zeta$ is independent of $y'$. The meaning of the singular functions is a superposition of such functions with discrete asymptotics with exponents $r^{-z}$ for $z\in K$, and $\zeta$ is just the ``density'' of the superposition. Then, taking into account what we obtained in the constant discrete case the Sobolev regularity which is determined by the occurring $[\eta]$-powers together with the ${\hat{H}}^s\big({\R}^q_{\eta},C^{\y}(X)\big)$-valued character of $\hat{\zeta}$ is nothing else than
\beq\label{contreg}
\underset{z\in K}{\inf}\big(s+\reb\,z-\varepsilon-(n+1)/2\big)
\eeq
for any $\varepsilon>0$.

Let us now draw some conclusions of Theorem \ref{reg} on a way to approximate the singular functions of branching edge asymptotics by singular functions of continuous asymptotics belonging to  a decomposition of the considered compact set $K=\bigcup_{i=0}^NK_i$, where the $K_i$ are as in \eqref{Ki}. The decomposition \eqref{BE} may also be applied to the $E^s$-valued case, cf. \eqref{Es}, i.e., we can write \eqref{Del} in the form
\beq\label{decd}
{\hat{\delta}}(y,\eta)=\sum_{i=0}^N\varphi_i(y){\hat{\delta}}_i(y,\eta)
\eeq
for summands ${\hat{\delta}}_i(y,\eta)\in{\Scal}\big(\R^q, {\Acal}'(K_i,E^s)\big)^{\bullet}$ (the Schwartz function is taken for convenience; it does not affect the results). The space ${\Scal}\big(\R^q, {\Acal}'(K_i,E^s)\big)^{\bullet}$ is closed in ${\Scal}\big(\R^q, {\Acal}'(K_i,E^s)\big)$.
 Let $B_i$ denote the analogue of the operator $B$ in the Theorem \ref{reg} now referring to $K_i$, i.e.,
$B_i:{\Scal}\big(\R^q, {\Acal}'(K_i,E^s)\big)\to{\Acal}'(K_i,E^s)$.
Then, applying $B_i$ to ${\hat{\delta}}_i(y,\eta)\in{\Scal}\big(\R^q, {\Acal}'(K_i,E^s)\big)^{\bullet}$ we obtain an element
\beq\label{decd1}
\hat{\chi}(y,\eta):=\sum_{i=0}^N\varphi_i(y)B_i{\hat{\chi}}_i(y,\eta)
\eeq
which is now a kind of approximation of the branching pointwise discrete functional $\hat{\delta}(y,\eta)$ by ${\hat{\chi}}(y,\eta)$ which turns the asymptotics to a continuous behaviour over $K_i$ where $y$ varies over $U_i$. Since by Theorem \ref{reg} the singular functions associated with $\hat{\delta}(y,\eta)$ and $\hat{\chi}(y,\eta)$ remain the same, we obtain the following Sobolev regularity approximation of the coefficients in the singular functions of branching edge asymptotics.
\begin{cor}\label{Sobreg}
Consider the branching discrete functional $\hat{\delta}(y,\eta)$ and the associated singular functions $$
F^{-1}_{\eta\to y}\{[\eta]^{(n+1)/2}\omega(r)\la{\hat{\delta}}(y,\eta),(r[\eta])^{-z}\ra\}.
$$
Then according to \eqref{decd} we may replace $\hat{\delta}(y,\eta)$ by the finite sum \eqref{decd1}, and from \eqref{contreg} we obtain the Sobolev regularity in the edge variables $y\in U_i$, namely,
$$
\underset{z\in K_i}{\inf}\big(s+\reb\,z-\varepsilon-(n+1)/2\big)
$$
for any $\varepsilon>0, i=0,1,\ldots,N$.
 In other words the Sobolev regularity may be localized over $U_i$ for the corresponding $K_i$, and, of course, the diameters both of $U_i$ and $K_i$ may be chosen as small as we want when we choose $N$ sufficiently large.
\end{cor}
In other words, if we apply Theorem \ref{reg} to a $\hat{\delta}(y,\eta)\in{\Scal}\big(\R^q,{\Acal}'(K,E^s)\big)^{\bullet}$ with variable in $y$ and in general branching patterns of $y$-wise discrete asymptotics, then ``intuitively'' the Sobolev regularity at a point $y\in\R^q$ has the form \eqref{regass}, now for $p=p(y)$. Clearly the Sobolev smoothness in correct form refers to an open set in the $y$-variables. But Corollary \ref{Sobreg} tells us how to collapse such open sets to a single point, and then the Sobolev smoothness itself appears variable and branching under varying $y$.

Note that also the general continuous asymptotics carried by a compact set $K$ can be interpreted  in terms of decompositions into ``small'' parts $K_i$ when we write $K=\sum_{i=0}^NK_i$.  This allows us to read off the ``content'' of Sobolev regularity of singular functions as in Proposition \ref{aa} from the summands coming from $K_i$, and then we  have similar relations as in Corollary \ref{Sobreg}.

\end{document}